\documentclass[11pt,leqno,amscd,amssymb,verbatim, url]{amsart}

\usepackage{amscd,amssymb,verbatim,url,amsfonts,amssymb,amsthm,wrapfig}

\newtheorem{theorem}{Theorem}[section]
\newtheorem{lemma}[theorem]{Lemma}
\newtheorem{prop}[theorem]{Proposition}
\newtheorem{cor}[theorem]{Corollary}
\newtheorem{hyp}[theorem]{Hypotheses}

\theoremstyle{definition}
\newtheorem{definition}[theorem]{Definition}

\theoremstyle{remark}
\newtheorem{remark}[theorem]{Remark}
\newtheorem{conditions}[theorem]{Conditions}

\numberwithin{equation}{section}

\newcommand{\wF}{{\widetilde F}}

\newcommand{\wC}{{\widetilde C}}
\newcommand{\wE}{{\widetilde E}}

\newcommand{\wrad}{{\widetilde\rad}}
\newcommand{\fb}{{\mathfrak b}}
\newcommand{\wfb}{{\widetilde{\mathfrak b}}}

\newcommand{\wAmod}{\widetilde A{\text{\rm -mod}}}

\newcommand{\wfa}{{\widetilde{\mathfrak a}}}
\newcommand{\gr}{\text{\rm gr}\,}

\newcommand{\sB}{{\mathcal {B}}}
\newcommand{\sC}{{\mathcal {C}}}

\newcommand{\Imm}{{\text{\rm Im}}\,}
\newcommand{\sA}{{\mathcal A}}

\newcommand{\Spec}{{\text{\rm Spec}}}

\newcommand{\Ext}{{\text{\rm Ext}}}

\newcommand{\x}{\text{\bf \it x}}

\newcommand{\fa}{{\mathfrak a}}

{

\newcommand{\Amod}{A\mbox{--mod}}

\newcommand{\Hom}{\text{\rm Hom}}

\newcommand{\End}{\operatorname{End}}

\newcommand{\sO}{{\mathcal{O}}}

\newcommand{\rad}{\operatorname{rad}}

\newcommand{\wM}{{\widetilde{M}}}
\newcommand{\wA}{{\widetilde{A}}}
\newcommand{\wB}{{\widetilde{B}}}
\newcommand{\wR}{{\widetilde{R}}}
\newcommand{\wU}{{\widetilde U}}
\newcommand{\wDelta}{{\widetilde{\Delta}}}

\newcommand{\wN}{{\widetilde{N}}}
\newcommand{\wD}{{\widetilde{D}}}
\newcommand{\wP}{{\widetilde{P}}}
\newcommand{\wQ}{{\widetilde{Q}}}
\newcommand{\wJ}{{\widetilde{J}}}

\newcommand{\wL}{{\widetilde{L}}}

\newcommand{\blist}{\begin{list}{\rom{(\roman{enumi})}}{\setlength
{\leftmarg in}{0em} \setlength{\itemindent}{7ex}
\setlength{\labelsep}{2ex}\setlength{\listparindent}{\parindent}
\usecounter{enumi}}}
\newcommand{\elist}{\end{list}}

\begin{document}

 \title[Forced Gradings in Integral Quasi-hereditary Algebras]{Forced Gradings in Integral Quasi-hereditary  Algebras \\ with Applications to Quantum Groups}


\author{Brian J. Parshall}
\address{Department of Mathematics, University of Virginia, Charlottesville, VA 22903} 
\email{bjp8w@virginia.edu}
\thanks{Research supported in part by the National Science
Foundation DMS-1001900}

\author{Leonard L. Scott}
\address{Department of Mathematics, University of Virginia, Charlottesville, VA 22903} 
\email{lls2l@virginia.edu}

\subjclass[2000]{Primary 17B55, 20G; Secondary 17B50}

\date{}

\begin{abstract}
 Let $\sO$ be a discrete valuation ring with fraction field $K$ and residue field $k$. A quasi-hereditary algebra $\wA$ over $\sO$  provides a bridge between the representation theory of the
quasi-hereditary algebra $\wA_K:=K\otimes \wA$ over the field $K$ and the quasi-hereditary algebra 
$A_k:=k\otimes_\sO\wA$ over $k$.  In one important example, $\wA_K$--mod is a full subcategory of
the category of modules for a quantum enveloping algebra while $\wA_k$--mod is a full subcategory of
the category of modules for a reductive group in positive characteristic. This paper considers first the question
of when the positively graded algebra $\gr \wA:= \bigoplus_{n\geq 0}(\wA\cap\rad^n\wA_K)/(\wA\cap\rad^{n+1}\wA_K)$ is quasi-hereditary. A main result gives sufficient conditions that $\gr\wA$ be quasi-hereditary. The
main requirement is that each graded module $\gr\wDelta(\lambda)$ arising from a $\wA$-standard (Weyl) module $\wDelta(\lambda)$ have an irreducible head. An additional hypothesis requires that the graded algebra $\gr \wA_K$ be quasi-hereditary, a property recently proved \cite{PS9} to hold in some important cases involving quantum enveloping algebras. 
 In the case where $\wA$ arises from regular dominant weights for a 
 quantum enveloping algebra at a primitive $p$th root of unity
for a prime $p>2h-2$ (where $h$ is the Coxeter number), a second main result shows that $\gr\wA$ is quasi-hereditary. The proof of
this difficult result involves interesting new methods involving integral quasi-hereditary algebras. It
also depends on previous work \cite{PS9} of the authors, including a continuation of the methods there
involving tightly graded subalgebras, and a development of a quantum deformation theory over $\sO={\mathbb Z}_{(p)}(\zeta)$, where $\zeta$ is a $p$th root of unity. This (new) deformation theory, given in Appendix II (\S7), extends work of Andersen-Jantzen-Soergel \cite{AJS}, and is worthy of attention in its own right.
 As we point out, the present paper provides an essential
step in our work on $p$-filtrations of Weyl modules for reductive algebraic groups over fields of positive
characteristic.
\end{abstract}

\maketitle


\section{Introduction} Quasi-hereditary algebras are certain finite dimensional
algebras over a field which arise naturally
in Lie-theoretic representation theory, in both geometric and algebraic contexts. For example, if $G$ is a reductive algebraic group in positive
characteristic, the category of finite dimensional rational $G$-modules generated by the irreducible
modules having highest weights in a finite saturated set of dominant weights is equivalent to the
module category $\Amod$ for a quasi-hereditary algebra $A$; see \cite{CPS-1}. In a geometric vein,
certain categories of perverse sheaves are similarly equivalent to $\Amod$ for a quasi-hereditary
algebra $A$; see \cite{PS0}. On the other hand, quasi-hereditary algebras are defined abstractly, so
they are
objects of interest to finite dimensional algebraists; see \cite{R}, for example.

In \cite{PS9}, the authors studied naturally occurring hypotheses which, given a quasi-hereditary algebra $A$,
guarantee a quasi-hereditary structure on the positively graded algebra
$$\gr A:=\bigoplus \rad^iA/\rad^{i+1}A.$$
 For
general quasi-hereditary algebras $A$, it is quite unexpected that $\gr A$ should also be quasi-hereditary, though \cite{PS9} demonstrated that this does occur in several situations of interest in the theory of quantum and algebraic
groups.

In \cite{CPS1a}, the notion of an {\it integral} quasi-herediary $\wA$ over a commutative ring
$\sO$ was introduced. See \S2 below for a summary. In this paper, $\sO$ is taken to be a DVR
with fraction field $K$ and residue field $k=\sO/(\pi)$. We take up the question of when, for the integral quasi-hereditary
algebra $\wA$ over $\sO$, the positively graded algebra
$$\gr\wA:=\bigoplus_{i\geq 0}\frac{\wA\cap \rad^i\wA_K}{\wA\cap\rad^{i+1}\wA_K}$$
is an integral quasi-hereditary algebra over $\sO$. The (integral) quasi-heredity of $\wA$ implies, for example, that the
$K$-algebra $\wA_K:=K\otimes_{\sO}\wA$  and the $k$-algebra $A=\wA_k:=\wA\otimes_{\sO}k$ are both quasi-hereditary. We will always assume that $\wA$ is associated with a fixed weight
poset $\Lambda$ indexing its irreducible modules, over either $K$ or $k$, and is ``split" (see \S2).
The algebra $\wA$ provides a link between the representation theory of the $K$-algebra $\wA_K$
and the $k$-algebra $A$. 

Of course, if $\gr\wA$ is quasi-hereditary over $\sO$, then $\gr\wA_K$ is also quasi-hereditary (with the
same weight poset $\Lambda$ as $\wA$, and with standard modules $\gr\wDelta(\lambda)_K$)---already a
difficult property to prove. We simply assume, as a starting point, that $\gr\wA_K$ is quasi-hereditary. In fact, in the applications we have in mind, this assumption will be valid. The first
 main theorem, Theorem \ref{mainIntGrthm}, establishes that $\gr \wA$ is integral quasi-hereditary (with
the same poset as $\wA$ and standard modules $\gr\wDelta(\lambda)$) if and only if each $\gr\wDelta(\lambda)$, $\lambda\in\Lambda$, has a simple head. Here\footnote{These gr-constructions make sense for any integral domain and lattices for algebras over them, and, as far as we know, have
not been previously studied.}
$$\gr\wDelta(\lambda):=\bigoplus_{i\geq 0}\frac{\wDelta(\lambda)\cap\rad^i\wDelta(\lambda)_K}{
\wDelta(\lambda)\cap\rad^{i+1}\wDelta(\lambda)_K}.$$
Corollary \ref{Nameless-1} gives conditions under which  $\wA$-modules $\wM$ with standard filtrations give rise to
$\gr\wA$-modules $\gr\wM$ with standard filtrations.

In \S4, we show that the simple head hypothesis required in Theorem  \ref{mainIntGrthm} holds, provided certain conditions involving an integral subalgebra $\wfa$ of $\wA$ are satisfied.\footnote{ The use of such a special subalgebra has been previously explored by us in \cite{PS9}, and this section is
a natural extension of those results.}  The conditions are stated in  Conditions \ref{conditions} and
the hypotheses of  Theorem \ref{bigtheorem}. The latter theorem is
based on the (new) notion of a tight lattice (see Defintion \ref{tightdef1} just below).
All this is pulled together in Corollary \ref{cor4.4} and Theorem \ref{integralquasihereditary}.

In \S5, we consider the case in which the integral quasi-hereditary algebra $\wA$ arises as a suitable
``regular weight"
homomorphic image of a quantum enveloping algebra at a root of unity. (The regular blocks of the famous $q$-Schur algebras provide examples
of such algebras.) The first main result here is Theorem \ref{preliminaryresult} which verifies that Conditions \ref{conditions}
hold in case $p>h$ (the Coxeter number) and $\wA$ is associated to a finite ideal of regular weights. 
Then the main result in this section, given in Theorem \ref{mainquantumresult}, shows that, when
$p\geq 2h-2$, $\gr\wA$ is a quasi-hereditary algebra with standard modules $\gr\wDelta(\lambda)$. 
Its proof completes the verification process begun in Theorem \ref{preliminaryresult}
by checking that the hypotheses of Theorem \ref{bigtheorem} hold, so that all the
results of \S4 are available. 

A key step in the proof of Theorem \ref{preliminaryresult}, namely, the verification of Condition \ref{conditions}(5), is postponed
to Appendix II (\S6), entitled ``Quantum deformation theory over $\sO$." The appendix is inspired by, and
heavily uses, the deformation theory in the paper \cite{AJS}, in which Andersen-Jantzen-Soergel showed that
the 
(known to be true) Lusztig conjecture for quantum groups at a $p$th root of unity implies the characteristic
$p$ version of the conjecture for primes $p\gg 0$, depending on the root system. The paper \cite{AJS} presents
two parallel deformation theories, the first in positive characteristic for restricted Lie algebra representations,
and the second for the small quantum group at a root of unity $\zeta$. In our appendix, we extend the
results of \cite{AJS} to include a third deformation theory, for a natural integral form of the small quantum group over an appropriate DVR $\sO\subseteq{\mathbb Q}(\zeta)$ with $\zeta$ a $p$th root of unity. We require that $p$ be a prime $>h$. The main consequence of this theory is the demonstration of a positive grading
on the $\sO$-integral form which base changes to the grading on the small (``regular block") quantum
group over ${\mathbb Q}(\zeta)$
studied in \cite[\S17, \S18]{AJS}. See Theorem \ref{newandimportant}, which can also be regarded as
a main result of this paper. This is just what is needed to complete the verification of Condition
\ref{conditions}(5), as required in the proof of Theorem \ref{preliminaryresult}(c).  Observe that \cite[\S18]{AJS} is quoted earlier in checking Condition \ref{conditions}(1) (in  ``case 2" of \cite{AJS}). 

One area where the results of this paper are very useful is the study of filtrations of standard
modules both for $\wA$ and for $A=\wA_k$.  For example, an important question concerns the existence of
$p$-filtrations of Weyl modules $\Delta(\lambda)$ (even in the non-regular case) for a reductive group $G$, i.~e., filtrations with
sections of the form $L(\mu_0)\otimes\Delta(\mu_1)^{(1)}$ with $\mu_0$ a $p$-restricted dominant weight and
$\mu_1$ an arbitrary dominant weight. See \cite{A1}, though the terminology goes back earlier to an
1990 MSRI lecture of Donkin,
and the concept already appears in the work of Jantzen \cite{Jan}.  Of course, anything that can be said about filtrations of standard modules for $\gr\wA$ and $\gr A$
can be phrased as statements about filtrations of standard modules for $\wA$ and $A$.
The additional structure afforded by the gradings makes such filtration assertions even more interesting.
Indeed, \cite{PS10}, which uses our results here to present new results on $p$-filtrations of Weyl modules for semisimple groups from this perspective, using tools
unavailable in the ungraded (or non-integral) setting.

We also expect to use \cite{PS10} and the  Lusztig modular character conjecture to improve the Koszulity theory
given in \cite{PS9}. One unexpected feature of our results in the present paper is that the Lusztig
modular conjecture is not
required. We do use the root of unity quantum analog of the conjecture, but that is known to be a theorem for $p>h$ \cite{T}.

\section{Some preliminary notation/results}

\medskip
 
  Let $(K,{\mathcal{O}},k)$ is a $p$-local system, i.~e., ${\mathcal{O}}$ is a
discrete valuation ring with maximal ideal ${\mathfrak{m}}=(\pi)$, fraction field $K$, and residue field
$k={\mathcal{O}}/{\mathfrak{m}}$.\footnote{In the text, given a ring or algebra $R$, an $R$-module
$M$ is {\it finite} (or $R$-finite) if it is finitely generated as an $R$-module.} The field $k$ is allowed to have arbitrary characteristic; in our
applications it will have positive characteristic $p$.

Let $A_{\sO}$ be an ${\mathcal{O}}$-algebra, which will always be assumed to be an $\sO$-free module of finite rank. Base change to $K$ and $k$ defines algebras $A_{K}:=K\otimes_{\sO}A_{\sO}$ and
$A_{k}:=k\otimes_{\sO}A_{\sO}=A_{\sO}/\pi A_{\sO}$ over $K$ and $k$, respectively. It will be often convenient to denote $A_{\sO}$ by $\wA$ and denote $A_k$ simply by $A$.  More generally, if ${\widetilde{M}}$ is an ${\widetilde{A}}$-module,
write $\wM_{K}$ for $K\otimes_{{\mathcal{O}}}{\widetilde{M}}$ and $M$ for $\wM_k=k\otimes_{\sO}{\wM}\cong\wM/\pi\wM$; sometimes, we write $\overline
{{\widetilde{M}}}$ for $\wM_k$.  The $\wA$-module $\wM$ is said to be a $\wA$-lattice if it is $\sO$-finite
and torsion free. (This definition applies when $\sO$ is any integral domain; see Reiner
 \cite[pp.~44, 129]{Reiner}. Any $\sO$-finite $\wA$-submodule $\wM$ of a finite dimensional $\wA_K$-module $N$  is a 
$\wA$-lattice, and is said to be a full lattice in $N$ if also $K\wM=N$, in which case $K\otimes_\sO\wM
\cong N$. 
Let $\wAmod$ be the category of
$\sO$-finite $\wA$-modules. The category $\wAmod$ contains all $\wA$-lattices, and also contains $\Amod$ as a full subcategory.

It will often be very convenient, for a non-negative integer $n$, to denote the ideal $\wA\cap\rad^n\wA_K$ of $\wA$ by $\wrad^n\wA$.
Here $\rad^n\wA_K=(\rad\wA_K)^n$, the $n$th power of the radical (= maximal nilpotent ideal) of the
finite dimensional algebra $\wA_K$. Of course, if $n\geq 1$, $\wrad^n\wA$ is a nilpotent ideal in $\wA_K$ and $(\wrad^n\wA)_K=\rad^n\wA_K$. Define
\begin{equation}\label{gradedwA}
\gr\wA:=\bigoplus_{n\geq 0}\frac{\wA\cap \rad^n\wA_K}{\wA\cap\rad^{n+1}\wA_K}=\bigoplus_{n\geq 0}\wrad^n\wA/\wrad^{n+1}\wA.
\end{equation}
Necessarily the graded algebra $\gr\wA$ is a finite $\sO$-free module and $(\gr\wA)_K\cong\gr\wA_K=\oplus_{n\geq 0}\rad^n\wA_K/\rad^{n+1}\wA_K$.

A similar notation will be used for modules and lattices. For  any $\wA$-lattice $\wM$ and non-negative
integer $n$, let $\wrad^n\wM=\wM\cap\rad^n\wM_K$ (where $\rad^n\wM_K=(\rad^n\wA_K)\wM_K$), and set
\begin{equation}\label{gradedmodule}
\gr\wM:=\bigoplus_{n\geq 0}\frac{\wM\cap \rad^n\wM_K}{\wM\cap\rad^{n+1}\wM_K}=\bigoplus_{n\geq 0}\wrad^n\wM/\wrad^{n+1}\wM.\end{equation}
Then $\gr\wM$ is a graded $\gr\wA$-lattice, and $$(\gr\wM)_K\cong\gr \wM_K=\oplus_{n\geq 0}\rad^n\wM_K/\rad^{n+1}\wM_K.$$

There are one-to-one correspondences of isomorphism classes of irreducible modules:
\begin{equation}\label{correspondence1}\text{\rm Irr}((\gr\wA)_K) \leftrightarrow \text{\rm Irr}(\gr\wA_K)\leftrightarrow\text{\rm Irr}(\wA_K).\end{equation}
Also, $\wA\cap\rad\wA_K$ is an $\sO$-pure nilpotent ideal in $\wA$. Hence,
\begin{equation}\label{correspondence2}\text{\rm Irr}((\gr\wA)_k)\leftrightarrow \text{\rm Irr}(\gr\wA)\leftrightarrow \text{\rm Irr}((\wA/(\widetilde\rad\wA)_k))\leftrightarrow\text{\rm Irr}(\wA)
\leftrightarrow\text{\rm Irr}(A).\end{equation}
The relation between the sets $\text{Irr}(A_K)$ and $\text{Irr}(A)$ may be subtle (or non-existent); however,  in
the next section, we consider
a family of $\sO$-algebras $\wA$ in which the irreducible $A$-modules correspond bijectively, up to isomorphism, to the irreducible $\wA_K$-modules. Further, there is an evident
common indexing set $\Lambda$. If $\lambda\in\Lambda$, we will let  $\wL_K(\lambda)$ (resp., $L(\lambda)$) be the corresponding irreducible $\wA_K$-(resp., $A$-) module.

The following general lemma which will be needed in \S5.

\begin{lemma}\label{preliminarysectionlemma} Suppose $\fa_K\subseteq\fb_K\subseteq A_K$ are finite dimensional $K$-algebras. Assume that every irreducible $A_K$-module is absolutely irreducible and a direct sum of irreducible $\fb_K$-modules, each of which is absolutely irreducible over $\fa_K$. Then

(a) $\fb_K\cap\rad A_K=\rad\fb_K$, and $\fa_K\cap \rad A_K=\fa_K\cap \rad \fb_K=\rad\fa_K$.

(b) The algebras $\fa_K$ and ${\mathfrak b}_K$ (as well as $A_K$) are split semisimple. Any Wedderburn complement $\fa_{K,0}$ of $\fa_K$ is contained in a Wedderburn complement
$\fb_{K,0}$ of $\fb_K$, and every Wedderburn complement $\fb_{K,0}$ of  $\fb_K$ is contained in a
Wedderburn complement $A_{K,0}$ of $A_K$.\end{lemma}

\begin{proof} We first prove (a). Since $\fb_K\cap\rad A_K$ is a nilpotent ideal in $\fb_K$, we have
$\fb_K\cap\rad A_K\subseteq\rad\fb_K$. Since $\rad \fb_K$ kills every irreducible $A_K$-module, we also have
$\rad\fb_K\subseteq \rad A_K$, and so $\rad\fb_K\subseteq\fb_K\cap\rad A_K$. Thus, $\fb_K\cap\rad A_K=
\rad\fb_K$. A similar argument shows that $\fa_K\cap \rad A_K=\rad\fa_K$.

Since $\fa_K$ and $\fb_K$ are subalgebras of $A_K$, all their irreducible modules can be found as composition factors of irreducible $A_K$-modules. Thus,
 all irreducible $\fa_K$-modules are absolutely irreducible, and all irreducible $\fb_K$-modules are are absolutely irreducible. In particular, $\fa_K/\rad\fa_K$ and $\fb_K/\rad\fb_K$ are split semisimple algebras, as is
$A_K/\rad A_K$. By part (a) above, we have $\fa_K/\rad\fa_K\subseteq\fb_K/\rad\fb_K\subseteq A_K/\rad A_K$. Now (b) follows from standard separability arguments (vanishing of Hochschild 1- and 2-cohomology).
\end{proof}
The following general definition makes sense  when $\sO$ is any integral domain, though we will apply it only in our DVR context.

\begin{definition}\label{tightdef1} Let $\widetilde{\mathfrak b}$ be an $\sO$-finite and torsion free $\sO$-algebra, and let $\wM$ be a $\widetilde {\mathfrak b}$-lattice. Call $\wM$ is {\it tight} if
\begin{equation}\label{tightdef2}\widetilde\rad^r\wM =(\widetilde\rad^r{\widetilde{\mathfrak b}})\wM, \quad\forall r\in\mathbb N.\end{equation}
\end{definition}

Observe that, for any $\wfb$-lattice $\wM$,  the left-hand side always contains  the right-hand
side in (\ref{tightdef2}), and that equality holds for $\wM=\widetilde{\mathfrak b}$ or any finite projective $\widetilde
{\mathfrak b}$-module, which are, thus, tight lattices for $\widetilde{\mathfrak b}$.

Finally, this section concludes with the following well-known elementary lemma which will often be used without comment. Part (a) is \cite[Thm.~4.0]{Reiner}.  Part (b) gives one of the many characterizations of purity in the case of a DVR. We
leave the easy proof to the reader. 

\begin{lemma}\label{wellknownlemma} Let $R$ be an integral domain, and let $\wN$ be a submodule
of an $R$-finite, torsion free module $\wM$. (That is, $\wM$ is an $R$-lattice in the sense of
 \cite[p.~44]{Reiner} Then 

(a) $\wN$ is pure in $\wM$ (meaning that if $\wM/\wN$ is torsion free) if and only if $\wN=\wM\cap \wN_K$,
where $K$ is the fraction field of $R$; 

(b) if $R=\sO$ is a DVR with fraction field $K$, then $\wN$ is  
pure in $\wM$ if and only if the natural map
$\wN_k \to \wM_k$ obtained by base-change is injective. \end{lemma}

\section{Integral quasi-hereditary algebras}
We maintain the general notation from above, unless otherwise noted. Assume that $\wA$
is a ``split" quasi-hereditary algebra (QHA) over $\sO$  in the sense of \cite[Defn.~(3.2)]{CPS1a}. This definition requires that there exists a sequence
$0=\wJ_0\subset \wJ_1\subset\cdots\subset \wJ_t=\wA$ of ideals in $\wA$ such that $\wJ_i/\wJ_{i-1}$
     is an heredity ideal in $\wA/\wJ_{i-1}$ for $0<i\leq t$.\footnote{ An ideal $\wJ$ is heredity provided that
     (i) $\wA/\wJ$ is $\sO$-free; (ii) $\wJ^2=\wJ$; (iii) $\wJ$ is a projective left $\wA$-module; and (iv) $E=\End_\wA(\wJ)$ is semisimple over $\sO$. The split case, used here,  requires (over DVRs, such as $\sO$) the stronger, but easier to state, condition that $E$ is a direct sum of full matrix algebras $M_n(\sO)$ over $\sO$.}

      It follows that $\wA_K$--mod and $\Amod$ are  highest weight categories  (HWCs) \cite{CPS-1} with the same poset $\Lambda$ indexing
the (isomorphism classes of) irreducible modules $\widetilde L_K(\lambda)$ (resp., $L(\lambda)$) for $\wA_K$ (resp., $A$).
Indeed, by \cite[Prop.~2.1.1]{DS}, the category $\wAmod$ is an $\sO$-finite highest weight category in the sense of \cite[Defn.~2.1]{DS}.\footnote{ See also
\cite[Prop.~2.1.2, Rem./Defn.~2.1.3]{DS}, which explains the essential equivalence of the $\sO$-finite highest weight category notion (in the finite weight poset case) with that of a split
 QHA (over a commutative Noetherian ring $\sO$, which is a DVR here). If $\wAmod$ is a HWC with poset $\Lambda$, we often informally say
 that $\wA$ is a QHA with poset $\Lambda$.}
 The assumption that
$\wA$ is ``split" implies that the irreducible modules for $\wA_K$ and $A$ are are absolutely irreducible.  Also, the category $\wAmod$ has standard modules  $\wDelta(\lambda)$, $\lambda\in\Lambda$, which are $\wA$-lattices with the property that
the HWCs $\wA_K$-mod
(resp., $\Amod$) has standard objects $\{\wDelta_K(\lambda):=\wDelta(\lambda)_K\}_{\lambda\in\Lambda}$ (resp.,
$\{\Delta(\lambda)=\wDelta_k(\lambda):=\wDelta(\lambda)_k \}_{\lambda\in\Lambda}$). Similarly, $\wAmod$ has costandard objects $\widetilde \nabla(\lambda)$, $\lambda\in\Lambda$, etc. Each irreducible $A$-module $L(\lambda)$ has a (unique, up to isomorphism) projective cover $\wP(\lambda)$ in $\wAmod$; see \cite[Prop.~2.3]{DS} or \cite[p.~157]{CPS1a}\footnote{The arguments in \cite{DS} and \cite{CPS1a} do not require
   that $\sO$ be complete, or make any use of a completion of $\sO$.}, and $\wP(\lambda)$ has a $\wDelta$-filtration (i.~e., a filtration with sections
    isomorphic to $\wDelta(\mu)$, $\mu\in\Lambda$). In this filtration, $\wDelta(\lambda)$ appears as the top section, and all other sections $\wDelta(\mu)$ have
    $\mu>\lambda$. Clearly, $\wP(\lambda)$ is a $\wA$-lattice, and a direct summand of $\widetilde A$ (viewed as a module over itself). Moreover, the (left) $\wA$-module $\wA$
    decomposes a direct sum of various copies of $\wP(\lambda)$ (each appearing $\dim\,L(\lambda)$ times).

    Observe that $\gr\wP(\lambda)$ is the projective cover of $L(\lambda)$ in either $\gr\wAmod$ (ungraded module category) or in $\gr\wA$-grmod (graded module
    category). This claim follows from dimension considerations, since $(\gr\wP(\lambda))_0$, as a nonzero quotient of $\wP(\lambda)$, has $L(\lambda)$ as a homomorphic image. (Note head$\,\wA\cong\text{head}\,(\gr\wA)_0\cong\text{head}\,\wA$.)

    While $\wP(\lambda)_k$ is the projective cover $P(\lambda)$ in $\Amod$ of $L(\lambda)$, it
is not generally true that $P(\lambda)_K=\wP(\lambda)_K$ is the projective cover $P_K(\lambda)$ of $\wL_K(\lambda)$ in $\wA_K$--mod. However, we do have
\begin{equation}\label{decompostewPlambdatoK} \wP(\lambda)_K\cong P_K(\lambda) \oplus\bigoplus_{\mu>\lambda}P_K(\mu)^{\oplus m_{\lambda,\mu}},\end{equation}
for non-negative integers $m_{\lambda,\mu}$. This follows since $\Delta_K(\lambda)$ is a quotient of $\wP(\lambda)_K$, and since
all other standard modules $\Delta_K(\mu)$ appearing in any $\Delta_K$-filtration of $\wP(\lambda)_K$ have weights $\mu>\lambda$. Applying the
functor $\gr$ to (\ref{decompostewPlambdatoK}) gives a similar decomposition of $\gr(\wP(\lambda)_K)=(\gr\wP(\lambda))_K$, although we, as yet,
know little about $\gr\Delta_K$-filtrations of $\gr(\wP(\lambda))_K$ or of $P_K(\lambda)$. This issue will be addressed in Hypothesis \ref{hypothesis1} below.

\section{A main result}  This section approaches the question of determining  conditions on an integral quasi-hereditary algebra $\tilde A$ over $\sO$ which guarantee that the graded algebra $\gr \widetilde A$ as
defined in (\ref{gradedwA}) is quasi-hereditary. The first step is to find some properties $\widetilde A$ and $\gr\wA$ might ``obviously" share, at least in common situations. To this end we introduce a general framework of ``weight algebras". In this framework
we can replace the DVR $\sO$ by a general commutative ring $R$.  If ${\mathfrak p}\in\Spec R$ and if $M$ is an $R$-module, let $M_{\mathfrak p}$ denoted the $R_{\mathfrak p}$-module obtained by localizing $M$ at $\mathfrak p$.
Finally set $M({\mathfrak p}):=M_{\mathfrak p}/{\mathfrak p}M_{\mathfrak p}$, a vector space over
the residue field $R({\mathfrak p}):=R_{\mathfrak p}/{\mathfrak p}R_{\mathfrak p}$.

Let $B$ be an $R$-finite algebra. We call $B$ a {\it weight algebra}, if there is a set $\{e_\nu\,|\,\nu \in X\}$ of orthogonal idempotents in $B$, summing to 1, such that, for each ${\mathfrak p}\in {\text{\rm Spec}}\, R$, some subset $X_{\mathfrak p}$ of $X$ bijectively indexes the irreducible $B_{\mathfrak p}$-modules in such a way that, if
$\nu\in X_{\mathfrak p}$ corresponds to the irreducible $B_{\mathfrak p}$-module $L_{\mathfrak p}(\nu)$, then
$e_\nu L_{\mathfrak p}(\nu)\not=0$.  Somewhat loosely, we refer to $ X$ as the set of  ``weights" of $B$.
If $M$ is a $B$-module and $\nu\in X$, we write $M_\nu:=e_\nu M$, the $\nu$-{\it weight space} of $M$.

In this paper, we only consider weight algebras $B$ as above for which the sets $X_{\mathfrak p}$
can be uniformly chosen to be the same subset $\Lambda$ of $X$. In this case, we say $B$ is $\Lambda$-{\it uniform}. The weights in $X\backslash\Lambda$ then becomes largely irrelevant. In particular, we note
the following.

\begin{prop} Suppose that $B$ is a $\Lambda$-uniform weight algebra with weight set $X$. Then $B$ is
Morita equivalent to a $\Lambda$-uniform weight algebra with weight set $\Lambda$.
\end{prop}

\begin{proof} Put $P=\bigoplus_{\lambda\in\Lambda}Be_\lambda$. Then $P$ is a projective generator because of the non-vanishing conditions
$e_\lambda L_{\mathfrak p}(\lambda)\neq 0$ above. Thus, $B':=\bigoplus_{\lambda,\mu\in\Lambda} e_\lambda Be_\mu\cong(\End_B P)^{\text{\rm op}}$ is $R$-finite as a module and is Morita equivalent to $B$.   The idempotents $e_\lambda$, $\lambda\in\Lambda$, all belong to $B'$, and sum to the identity in
$B'$.  It is easily checked that these idempotents, indexed by $\Lambda$,
give $B'$ the structure of a $\Lambda$-uniform weight algebra with weight set $\Lambda$. \end{proof}

Some useful concepts can be defined in the generality of $\Lambda$-uniform weight algebras $B$. Let $\Gamma$ is a nonempty subset of $\Lambda$. Let $N$ be a finite $B$-module. Let $N_\Gamma $ be the largest quotient module for which, given $\nu\in\Lambda$, $e_\nu
N_\Gamma\not=0$ implies $\nu\in\Gamma$. Explicitly, $N_\Gamma=N/\sum_{\nu\in\Lambda\backslash\Gamma}Be_\nu N$. Considering $B$ as a left module over itself, it follows that $B_\Gamma$ is an $R$-algebra, and $N_\Gamma$ is a $B_\Gamma$-module. The notation agrees with the terminology at
the end of Example 2.1. In case $B=\wA$ is a split QHA over $\sO$ with weight poset $\Lambda$, then,
in case $\Gamma$ is an ideal in $\Lambda$, $\wA_\Gamma$ is an integral QHA with poset $\Gamma$ and
standard modules $\wDelta(\gamma)$, $\gamma\in\Gamma$ (the standard modules of $\wA$ indexed by
elements in $\Gamma$).

The following definition introduces an even stronger property for weight algebras which will hold in
our applications. The definition is somewhat redundant, in that conditions (1) and (2) on the set
$\Lambda$ of weights for $B$ imply already that $B$ is $\Lambda$-uniform, and give a ``natural" indexing of irreducible modules.

\begin{definition} Let $B$ a $\Lambda$-uniform weight algebra. Then $B$ is $\Lambda$-standard if the
following conditions hold:

\begin{itemize}
\item[(1)] $\dim L_{\mathfrak p}(\lambda)_\lambda=1$ for all $\lambda\in\Lambda$ and all  $\mathfrak p\in\,{\text{\rm Spec}}\, R$. (Here the dimension is computed over the residue field $R({\mathfrak p})$ of $R_{\mathfrak p}$.)

\item[(2)] There is a poset ordering $\leq$ on $\Lambda$ such that, for all ${\mathfrak p}\in\text{\rm Spec}\,R$, and all $\lambda,\mu\in\Lambda$,  if $L_{\mathfrak p}(\lambda)_\mu\not=0$, then $\mu\leq\lambda$.
\end{itemize}\end{definition}

When $R=\sO$, then $\text{\rm Spec}\,\sO=\{ {\mathfrak m},0\}$, with $\sO(0)=K$ and $\sO({\mathfrak m})=k$.

  The following lemma, which applies for any Noetherian integral domain $R$, is one of the main points of the ``weight algebra" development of this section. The proof is essentially obvious, noting that the ideal $B\cap \rad B_K$ and its $\gr B$ counterpart are both nilpotent ideals. The Noetherian hypotheses on $R$ insures
that $\gr B$ is finite $B$-module.

\begin{lemma} Assume that $B$ is a weight algebra over a Noetherian integral domain $R$ with fraction
field $K$, and put $\gr B=\bigoplus_{n\geq 0}\frac{\widetilde\rad^nB}{\widetilde\rad^{n+1}B}$.
If $\{e_\mu\,|\,\mu\in X\}$ is a set of idempotents giving $B$ the structure of a weight algebra over $R$,
the ``same" set $\{e_\mu\,|\,\,u\in X\}\subseteq(\gr B)_0\subseteq \gr B$ gives $\gr B$ the structure
of a weight algebra. If $\Lambda$ is a subset of $X$, then $B$ is $\Lambda$-standard if and only if $\gr B$
is $\Lambda$-standard.
\end{lemma}

Now return to the situation in which $R=\sO$ is a DVR as above. Let $\wA$  be a $\Lambda$-standard weight algebra over $\sO$, and suppose that $\wN$ is an $\wA$-lattice.    If $\lambda\in\Lambda$,
let $\wN_K'(\lambda)$ denote the $\wA_K$-submodule of $\wN_K$ generated by all $\mu$-weight vectors in $\wN_K$ with
$\mu\in\Lambda$ and $\mu>\lambda$. An element $v\in\wN_\lambda$ is called $\lambda$-{\it primitive}\footnote{This notion is inspired by Xi \cite{Xi}, but is stronger, even for $\wA_K$,
 than   the definition of ``primitive" used there.  Notice, if $v\in\wN_\lambda\subseteq \wN_{K,\lambda}$ is primitive in our sense, then $\lambda$ must
    be maximal among $\mu\in\Lambda$ such that $(\wN_K/\wN'_K(\lambda))_\mu\not=0$.} in $\wN$ if the image of
$\sO v$ in $\wN/\wN\cap N'_K(\lambda)$ is pure and nonzero. (Equivalently,\footnote{This equivalence and its ``strongly $\lambda$-primitive" analogue below require that $\sO$ be a DVR, though in most of the discussion, including Proposition \ref{stronglyprimitive} below, it is only necessary that
$\sO$ be an integral domain.} $v$ has nonzero image in $(\wN/\wN\cap \wN'_K(\lambda))_k$.)

We say that $v\in\wN_\lambda$ is strongly $\lambda$-primitive if $v$ represents a homogeneous element
$[v]$ in $\gr \wN$ such that the image of $\sO [v]$ in $\gr \wN/\wN\cap \wN'_K(\lambda)$ is pure and nonzero.  (Equivalently, $[v]$ has nonzero image in $(\gr\wN/\wN\cap \wN'_K(\lambda))_k$.) If $v$ is
strongly $\lambda$-primitive, then it is $\lambda$-primitive; see the proposition below. Necessarily, $v$ and, of
course, $[v]$ are non-zero. If $v$ is strongly $\lambda$-primitive, we also say that $[v]$ is strongly $\lambda$-primitive.

An element $v\in\wN$ is said to be primitive (resp., strongly primitive) if it is $\lambda$-primitive (resp., strongly $\lambda$-primitive) for some $\lambda$.   The astute reader will observe that ``primitive" can be defined for $\gr\wN$ or any $\gr \wA$-lattice, and that every strongly primitive element of $\gr\wN$ is both
primitive and homogeneous.  For some well-behaved lattices, the notions coincide; that is, the primitive
homogeneous elements of $\gr\wN$ are strongly primitive.  See Remark \ref{bigremark2}(c).  It is useful
to use the technical translation of the ``strongly primitive" notion in (b) of the result below.

\begin{prop}\label{stronglyprimitive} Let $\wA$ be a $\Lambda$-standard weight algebra over $\sO$ and let $\wN$ be a $\wA$-lattice. Let $\lambda\in\Lambda$ and $v\in\wN_\lambda$.

(a) The element $v$ is $\lambda$-primitive if and only if $v\not\in \wN\cap\wN'_K(\lambda)$ and $\sO v +
(\wN\cap\wN'_K(\lambda))$ is $\sO$-pure.

(b) The element $v$ is strongly $\lambda$-primitive
if and only if there is a non-negative integer $i$ such that $v\in\widetilde\rad^n\wN$, $v\not\in
\wN\cap(\rad^{i+1}\wN_K+\wN_K'(\lambda))$, and $\sO v + \wN\cap(\rad^{i+1}\wN_K+\wN'_K(\lambda))$ is
$\sO$-pure in $\wN$.

(c) If $v$ is strongly $\lambda$-primitive, it is $\lambda$-primitive. \end{prop}

\begin{proof} The element $v\in\wN_\lambda$ is $\lambda$-primitive if and only if the image of
$\sO v$ in $\wN/\wN\cap\wN_K'(\lambda)$ is nonzero and pure, i.~e., if and only if $v\not\in\wN\cap\wN'_K(\lambda)$, and the quotient
$$\frac{\wN}{\wN\cap\wN'_K(\lambda)}\Huge{\slash}\frac{\sO v +(\wN\cap \wN'_K(\lambda))}{\wN\cap\wN'_K(\lambda)}$$
is torsion free. Equivalently,
$$\wN/(\sO v +(\wN\cap\wN_K(\lambda))$$
is torsion free, which just means that $\sO v +(\wN\cap \wN'_K(\lambda))$ is pure in $\wN$.
This proves (a).

We now prove (b). The element $v\in\wN_\lambda$ represents a homogeneous element $[v]$ of grade $i$ in
$\gr\wN$ with nonzero image in $\gr(\wN/\wN\cap\wN'_K(\lambda))$ if and only if $v\in\wN\cap\rad^i\wN_K$ and
$v\not\in\rad^{i+1}\wN_K+ \wN_K'(\lambda)$ for some (uniquely determined) $i$. Purity of the image
of $\sO [v]$ in $(\gr(\wN/(\wN\cap\wN'_K(\lambda)))$ is equivalent to the purity of $\sO [v]$ in
$(\gr(\wN/(\wN\cap\wN'_K(\lambda)))_i$ for this integer $i$. The is equivalent to the
purity of
the sum $\sO v +(\wN\cap (\wN'_K(\lambda)+\rad^{i+1}\wN_K)$ in $\wN$. This proves (b).

Finally, to see (c), assume that $v$ is strongly $\lambda$-primitive, and choose the index $i$ as in the proof
of (b) above. Thus, $v\not\in\wF:=(\wN\cap(\wN'_K(\lambda)+\rad^{i+1}\wN_K)$, and $\sO v +\wF$ is pure
in $\wN$. Put $\wE:= (\wN\cap(\wN'_K(\lambda))\subseteq\wF$. Then $\sO v\cap \wF=0$, so
that the quotient
$$(\sO v+\wF)/(\sO v +\wE)\cong\wF/(\wF\cap (\sO v+\wE))=\wF/\wE$$
is torsion free. It follows that $\wN/(\sO v + \wE)$ is torsion free, which implies, by (a), that $v$
is $\lambda$-primitive. This proves (c).
\end{proof}

 The $\gr\wA$-submodule of $\gr\wN$ generated by all of its strongly
primitive elements is denoted ${\gr^\flat}\wN$ (or by ${\gr^\flat}_{\wA}\wN$ to emphasize the
dependence\footnote{The {\it set} of strongly primitive elements of $\gr\wN$ does not depend on
$\wA$, so long as $\wN$ is an $\wA$-lattice. (So, for example, we could use instead a quotient of $\wA$ by
an $\sO$-pure ideal acting trivially on $\wN$.)  However, the $\gr\wA$-submodule generated by these
strong primitive elements might depend on $\gr\wA$, and, thus, on $\wA$. (The dependence may sometimes be eliminated with strong hypotheses on $\wN$; see Remark \ref{bigremark2}(b).)} on $\wA$). Since strongly primitive elements in $\gr\wN$ are, by
definition, homogeneous, the $\gr\wA$-module  ${\gr^\flat}\wN$ is a graded submodule of $\gr\wN$.

Now suppose that $\wA$ is a split integral QHA over $\sO$ with poset $\Lambda$. We call $\wA$  {\it toral quasi-hereditary} (TQHA) if it is a $\Lambda$-standard weight algebra over $\sO$, using the given poset structure
on $\Lambda$. Many integral QHAs come naturally equipped with such a structure. In any case, it can usually be assumed, by passing to a Morita equivalent algebra. In particular, we note the following.

\begin{prop}\label{Morita2} Suppose that $\wA$ is a split integral QHA over $\sO$ with poset $\Lambda$. Then $\wA$
is Morita equivalent to a TQHA $\wB$ over $\sO$. In addition, the algebras $\gr\wA$ and $\gr\wA_K$
are Morita equivalent to the algebras $\gr\wB$ and $\gr\wB_K$, respectively.
\end{prop}
\begin{proof} Let $\wP(\lambda)$ be the PIM (projective indecomposable module) associated to
$\lambda\in\Lambda$, and put $\wP:=\bigoplus_{\lambda\in\Lambda}\wP(\lambda)$. Clearly, $\wP$ is
a finite projective generator for $\wA$, so that $\wE:=(\End_\wA(\wP))^{\text{\rm op}}$ is
Morita equivalent to $\wA$. If we take $\wP(\lambda)=\wA e_\lambda$ for orthogonal idempotents $\{e_\lambda\}_{\lambda\in \Lambda}$, then $\wE$ identifies with the algebra
$\bigoplus_{\lambda,\mu\in\Lambda}e_\lambda\wA e_\mu$. With this identification, $\{e_\lambda\}_{\lambda\in\Lambda}\subseteq\wE$ is a set of orthogonal idempotents summing to the identity $e:=\sum e_\lambda$
of $\wE$.
It is easily checked that these idempotents give $\wE=e\wA e$ the structure of a $\Lambda$-standard
weight algebra. (Use the equivalence $\wM\mapsto e\wM$ from $\wA$-mod to $\wE$-mod.)

Since $\wA e=\wP$ is a projective generator for $\wA$, we have $\wA e\wA=\wA$. This implies that
$$(e\rad\wA_K e)^n=e\rad^n\wA_Ke,\,\, n\in{\mathbb N},$$
from which the remaining Morita equivalences can be deduced.\end{proof}

\begin{lemma}\label{generation} Let $\wA$ be a TQHA over $\sO$. For $\lambda\in\Lambda$, the standard $\wA$-module
$\wDelta(\lambda)$ is generated by its $\lambda$-weight space $\wDelta(\lambda)_\lambda:=
e_\lambda\wDelta(\lambda)=\wDelta(\lambda)\cap\wDelta_K(\lambda)_\lambda$.\end{lemma}

\begin{proof} The head of $\wDelta(\lambda)$ is isomorphic to $L(\lambda)$.  Also, $\wDelta(\lambda)_\lambda$ is $\sO$-free of rank 1. The lemma follows from Nakayama's lemma. \end{proof}

 We will sometimes need to assume one or both of the following following hypotheses.

\begin{hyp}\label{hypothesis1} (1) $\wA$ is a TQHA over $\sO$ with poset
$\Lambda$.

\medskip   (2)  The
graded $K$-algebra
$\gr\wA_K$
is a QHA with the same
weight poset $\Lambda$ as $\wA_K$. For
each $\lambda\in\Lambda$, the standard modules in the HWC $\gr \wA_K$-mod of not necessarily graded $\gr\wA_K$-modules are the modules
$\gr \wDelta_K(\lambda)$.

\end{hyp}

\begin{lemma}\label{projectivegradedlemma} Assume that Hypothesis \ref{hypothesis1}(1) holds. Let $\wP$ be projective in $\wAmod$. Then ${\gr^\flat}\wP=\gr\wP$.\end{lemma}

\begin{proof}  First, it is useful to note that every finite projective $\wA$-module $\wP$ has the property that the positive grade terms of $\gr \wP$ are contained in $\rad(\gr\wP)$, the intersection of all
maximal submodules. This follows from the case of $\wP=\wA$ and the nilpotence of the ideal of positive
grade terms. In particular, $\wP$ and $P$ have the same head.  See also the discussion above (\ref{correspondence2}).

Now to prove the lemma, it suffices to treat the case $\wP=\wP(\lambda)$, $\lambda\in\Lambda$.  The
map $\wP(\lambda)_\lambda\to\wDelta(\lambda)_\lambda$ is surjective. So Lemma \ref{generation}
implies there exists an element $v\in\wP(\lambda)_\lambda$
whose image in $\wDelta(\lambda)_\lambda$ is a generator for $\wDelta(\lambda)$. The image of $v$ in $L(\lambda)={\text{head}}\wDelta(\lambda)$ is
nonzero, as is the image of $[v]$ in $L(\lambda)=\text{head}\,\gr\wP(\lambda)$. Thus, $[v]$ generates $\gr\wP(\lambda)$ by Nakayama.   It is, thus, sufficient to show that $v$ is strongly $\lambda$-primitive. However,
the image of $v$ in $\wP(\lambda)/\wP(\lambda)\cap(\rad\wP(\lambda)_K + \wP(\lambda)^\dagger_K(\lambda))$
is an $\wA$-generator of the latter nonzero module, hence generates an $\sO$-pure submodule.
 \end{proof}

 Assuming  Hypothesis \ref{hypothesis1}(1),  $\wDelta(\lambda)$ is generated by a strongly $\lambda$-primitive element
$v_\lambda\in\wDelta(\lambda)_\lambda$ and  there are no $\mu$-primitive elements in $\wDelta(\lambda)$ for $\mu\not=\lambda$. {\it Thus,
${\gr^\flat}\wDelta(\lambda)$ is generated by $ [v_\lambda]\in\gr\wDelta(\lambda)$, but we do not, in general, know that
$\gr\wDelta(\lambda)={\gr^\flat}\wDelta(\lambda)$,} although this does hold if $\lambda$ is maximal in $\Lambda$, by
Lemma \ref{projectivegradedlemma}.\footnote{We cannot just reduce to the maximal case without changing
the algebra $\gr\wA$ that is acting on $\gr\wDelta(\lambda)$. This is a subtle point, since the physical graded
$\sO$-module structure of the latter $\gr\wA$-module remains the same, if we attempt such a reduction.
However, the submodule generated by a given vector might change, as $\gr\wA$ changes}
 In favorable cases---see Lemma \ref{surjective}---if $\wR$ is a pure $\wA$-submodule of $\wN$, the map
 $\gr\wN\to\gr\wN/\wR$ induces a surjection
  $\gr^\flat\wN\to\gr^\flat\wN/\wR$, though we do not have independent conditions which guarantee surjectivity
for $\gr\wN\to\gr\wN/\wR$.  (The best we have is Corollary \ref{Nameless}, whose proof uses the surjection
$\gr^\flat\wN\to\gr^\flat\wN/\wR$ of Lemma \ref{preliminarylemma}, and requires additional hypotheses.

An immediate consequence of the discussion above is the following result.

\begin{lemma}\label{equivalence} Assume Hypothesis \ref{hypothesis1}(1). For $\lambda\in\Lambda$, $\gr \wDelta(\lambda)$ has a simple
head if and only if $\gr^\flat\wDelta(\lambda)=\gr\wDelta(\lambda)$.\end{lemma}

The next two results are technical results needed in Lemma \ref{mainlemma}.

\begin{lemma} \label{surjective} Assume Hypothesis \ref{hypothesis1}(1). Suppose $\wN$ is an $\wA$-lattice, and $\wR$ is an $\sO$-pure $\wA$-submodule.
 $\lambda\in\Lambda$, maximal with respect to $\wN_\lambda\not=0$,  such that
the following conditions hold.

\begin{itemize}
\item[(1)] $\wR\subseteq\wA_K\wN_\lambda$.

\item[(2)] $\wR_\lambda+(\wN\cap \rad^{i}\wN_K)_\lambda$ is pure in $\wN$ (or, equivalently, in $\wN_\lambda$),
for all $i\in\mathbb N$.

\item[(3)] Let $j_\lambda$ be the largest index $j$ such that  $\wN_\lambda\subseteq\rad^j\wN_K$. Then
$\wN_\lambda=\wR_\lambda\bigoplus(\widetilde\rad^{j_\lambda}\wN_\lambda).$\footnote{It can
be shown that (3) $\implies$ (2).}
\end{itemize}

Then the map $\gr\wN\to\gr\wN/\wR$ induces a surjection $\gr^\flat\wN\to\gr^\flat(\wN/\wR)$. \end{lemma}

\begin{proof} We first prove that each strongly primitive element in $\gr^\flat\wN/\wR$ is
the image of a strongly primitive element in $\gr\wN$.

Let $[v]\in(\gr\wN/\wR)_i$ be such a strongly primitive element, represented (through some abuse of notation) by
$v\in\wN$. We take $[v]$ to be $\mu$-strongly primitive for some $\mu\in\Lambda$.
Thus, for some grade $i$, $v\in \wN\cap(\rad^i\wN_K+\wR_K)$, the $\sO$-module  $\sO v +\wN\cap
(\rad^{i+1}\wN_K+\wR_K +\wN'(\mu)_K)$ is pure, and does not collapse to the right hand summand.  Without
loss, $v\in\wN_\mu\not=0$. Since $\lambda$ is maximal with $\wN_\lambda\not=0$, there are three cases to consider.

\smallskip\noindent{\underline{Case 1: $\lambda>\mu$}} Using condition (1), we have that $\wR_K\subseteq \wA_K\wN_K\subseteq\wN'(\mu)_K$.
Thus, $\wN\cap(\rad^{i+1}\wN_K+\wR_K+\wN'(\mu)_K)=\wN\cap(\rad^{i+1}\wN_K+ \wN'(\mu)_K)$.

Thus, the image of $v\in (\gr\wN)_i$ is $\mu$-strongly primitive, and maps to $[v]$ in $(\gr\wN/\wR)_i$.

\smallskip\noindent{\underline {Case 2: $\mu\not<\lambda$ and $\mu\not=\lambda$.}}  Then $(\wN'(\mu)_K)_\lambda=0$, so that, using condition (1), we get
$$\wN\cap(\rad^{i+1}\wN_K+\wR_K+\wN'(\mu)_K))_\mu=(\wN\cap(\rad^{i+1}\wN_K+\wN'(\mu)_K)_\mu$$
so that, again, $\sO v+\wN\cap(\rad^{i+1}\wN_K+\wN'(\mu)_K)$ is pure and does not collapse to the right-hand
summand.  Again, the image of $v$ in $(\gr \wN)_i$ is strongly primitive, and maps to $[v]$ in $(\gr\wN/\wR)$.

\smallskip\noindent{\underline{Case 3: $\mu=\lambda$.}} Here $\wN'(\mu)=\wN'(\lambda)=0$ by the maximality of
$\lambda$. For each $j\in\mathbb N$, $\wR_\lambda+(\widetilde\rad^i\wN)_\lambda$ is a lattice in
$(\wR_K + \rad^j\wN_K)_\lambda$, and is pure in $\wN$ by condition (2) . Hence,
$$(\wN\cap(\rad^j\wN_K+\wR_K))_\lambda=\wR_\lambda+(\widetilde\rad^j\wN)_\lambda,\quad
j\in\mathbb N.$$
 We claim that
$i>j_\lambda$. If not, then $\wN_\lambda\subseteq\wR_\lambda+(\wN\cap\rad^{i+1}\wN_K)$ by condition  (3). Thus,
$$\sO v+(\wR_\lambda +\wN\cap\rad^{i+1}\wN_K)=\wR_\lambda+\wN\cap\rad^{i+1}\wN_K,$$
contradicting the strong primitivity of $[v]$ in $\gr\wM/\wR$, taking $j=i$ in the previous display.

Thus, we may assume that $i>j_\lambda$. Notice $[v]=[v']$ if $v'\in v+\wR_\lambda$ and, of course,
the expression $\sO v + (\wR+\rad^{i+1}\cap\wN_K)$ remains unchanged if $v'$ replaces $v$. Using
condition (3), we may now assume that $v\in\wN_\lambda\cap\rad^{j_\lambda+1}\wN_K$.
$$\begin{aligned} (\wN_\lambda\cap\rad^{j_\lambda+1}&\wN_K)\cap(\sO v+\wR+\wN\cap\rad^{i+1}\wN_K)_\lambda\\
&=\sO v +\wR_\lambda\cap\rad^{j_\lambda+1}\wN_K +\wN_\lambda\cap\rad^{i+1}\wN_K
\\ &=\sO v+(\wN\cap\rad^{i+1}\wN_K).\end{aligned}$$
So the image of $v$ in $(\gr\wN)_i$ is strongly $\lambda$-primitive, and maps to to $[v]\in(\gr\wN/\wR)_i$.
This completes the proof in all cases that $[v]$ is the image of a strongly primitive element of $\gr\wN$.

It remains to prove that the image of $\gr^\flat\wN$ in $\gr\wN/\wR$ is contained in $\gr^\flat\wN/\wR$. For
this it is sufficient to show that the image in $\gr \wN/\wR$ of a strongly primitive element in $\gr \wN$ is
either strongly primitive or zero.

Suppose that $u\in\wN_\mu\cap\rad^i\wN_K$ represents a strongly primitive element $[u]$ in
$(\gr\wN)_i$, and suppose that the image of $[u]$ in $(\gr\wN/\wR)_i$ is not zero. There are again three
cases, as above, depending on the relationship of $\mu$ to $\lambda$.

If $\mu\not=\lambda$, we may use the formulas developed in Cases 1 and 2 above to prove the image of $[u]$ in $\gr(\wN/\wR)$ is strongly primitive. If $\mu=\lambda$ and $i>j_\lambda$, condition (3) implies the image of $[u]$ in
$\gr\wN/\wR$ is zero, as is, essentially argued in the discussion of Case 3. This completes the proof of
the lemma.
\end{proof}

A {\it prestandard module} of weight $\lambda\in\Lambda$ is a graded $\gr\wA$-module $\wD(\lambda)$ which is a submodule of
 $\gr\wDelta(\lambda)$, and which satisfies the condition that $(\gr\wDelta(\lambda))_{\lambda}=\wD(\lambda)_\lambda$. This implies that
 $\wD(\lambda)$ is a full $\gr\wA$-lattice in $\gr\Delta_K(\lambda)$. Equivalently, $\wD(\lambda)$ is a
graded $\gr\wA$-module satisfying
${\gr^\flat}\wDelta(\lambda)\subseteq\wD(\lambda)\subseteq\gr \wDelta(\lambda)$. For example, both ${\gr^\flat}\wDelta(\lambda)$ and $\gr\wDelta(\lambda)$
are prestandard modules of weight $\lambda$. The notation ${\gr^\flat}\wDelta(\lambda)$ depends on the algebra $\wA$. Thus, if $\wA'=\wA/\wJ$ is a natural
quasi-hereditary quotients associated with the poset ideal $\Gamma$ containing $\lambda$, then $\gr\wDelta(\lambda)$ is a module for $\gr\wA'$, but
the analogue ${\gr^\flat}_{\wA'}\wDelta(\lambda)$ of ${\gr^\flat}\wDelta(\lambda)$ constructed for $\wA'$, rather than $\wA$, may be larger (because the map $\gr\wA\to\gr\wA'$ may not
be surjective). However, if $\wD(\lambda)\subseteq\wDelta(\lambda)$ is prestandard with respect to $\gr\wA$, and is also a $\gr\wA'$-module, then
it remains  prestandard and still satisfies the required sandwich property. We have
$${\gr^\flat}\wDelta(\lambda)\subseteq\gr^\flat_{\wA'}\wDelta(\lambda)\subseteq\wD(\lambda)\subseteq\gr_{\wA'}\wDelta(\lambda)=\gr\wDelta(\lambda).$$
This is because ${\gr^\flat}_{\wA'}\wDelta(\lambda)$ is generated over $\gr\wA'$ by $(\gr\wDelta(\lambda))_\lambda$. In particular, ${\gr^\flat}_{\wA'}\wDelta(\lambda)$ is
itself prestandard. Later, we introduce much stronger hypotheses which guarantee that ${\gr^\flat}\wDelta(\lambda)=\gr\wDelta(\lambda)$, so that
all of this structure collapses. But, for now, we will keep track of it.

\begin{lemma}\label{preliminarylemma} Assume Hypothesis \ref{hypothesis1}(1). Suppose that $\wE$ is a graded $\gr\wA$-module satisfying
$$\bigoplus_{i=1}^n{\gr^\flat}\wDelta(\lambda_i)(s_i)\subseteq\wE\subseteq\bigoplus_{i=1}^n\gr\wDelta(\lambda_i)(s_i)$$
for $\lambda_1,\cdots,\lambda_n\in\Lambda$, $s_1,\cdots, s_n\in\mathbb Z$.
Then there are prestandard modules $\wD(\lambda_i)$ such that  $\wE$ has a filtration $0=\wF_0\subseteq\wF_1\subseteq\cdots\subseteq\wF_n=\wE$
 with sections $\wF_i/\wF_{i-1}\cong \wD(\lambda_i)(s_i)$
with $\lambda_i$ and $s_i$ as above, $i=1,\cdots, n$.\footnote{A different ordering of the weights $\lambda_1,\cdots,\lambda_n$ could conceivably
 result in different prestandard modules $\wD(\lambda_i)$.} (It may be that $\wD(\lambda_i)$ is distinct from $\wD(\lambda_j)$, even if $\lambda_i=\lambda_j$.)
\end{lemma}

\begin{proof}The proof is an easy induction, using projections, on the number $n$ of summands.\end{proof}

Now we are ready to prove the main lemma. For convenience, we assume both parts of Hypotheses \ref{hypothesis1}.

\begin{lemma}\label{mainlemma} Assume Hypothesis \ref{hypothesis1}.
Let $\wN$ be an $\wA$-lattice which has a $\wDelta$-filtration. Assume also that $\gr \wN_K
=(\gr\wN)_K$ has a $\gr\Delta_K$-filtration as a graded $\gr\wA_K$-module.  Then ${\gr^\flat}\wN$
has a graded filtration with sections graded $\gr\wA$-modules which have the form $\wD(\lambda)(s)$, $\lambda\in\Lambda$ and $s\in{\mathbb N}$,
for prestandard modules $\wD(\lambda)$. (Different sections can be associated to non-isomorphic prestandard modules of the same weight $\lambda$.)
 Applying $K\otimes_{\sO}-$ gives a graded
$\gr\Delta_K$-filtration of $\gr \wN_K$. (In particular, ${\gr^\flat}\wN$ is a full lattice in $\gr \wN_K$.) Moreover, we may choose the filtration
$$0=\wF_0\subseteq \wF_1\subseteq\cdots\subseteq\wF_r={\gr^\flat}\wN$$
so that , if $\lambda<\lambda'$ belong to $\Lambda$ and $\wF_i/\wF_{i-1}\cong
\wD(\lambda)(s)$, $\wF_j/\wF_{j-1}\cong\wD(\lambda')(s')$, for some $i,j,s,s'$, then
$i>j$.
\end{lemma}

\begin{proof} Since $\wN$ has a $\wDelta$-filtration, it contains a pure submodule $\wM$
such that $\wM$ is a direct sum of copies of $\wDelta(\lambda)$, for some
$\lambda\in\Lambda$ such that all weights of irreducible sections of $\wN/\wM$ are smaller than $\lambda$
or not comparable to it. The module $\wN/\wM$ also has a $\wDelta$-filtration.

Put
$$\gr^\#\wM=\bigoplus_{n\geq 0}\frac{\wM\cap(\rad\wA_K)^n\wN_K}{\wM\cap
(\rad\wA_K)^{n+1}\wN_K}.$$
There is an obvious inclusion map $\iota:\gr^\#\wM\to\gr\wN$ of graded $\gr\wA$-modules,
and the image $\Imm\iota$ is the kernel of the natural map $\gr\wN\to
\gr\wN/\wM$. (We do not claim this latter map is surjective.)
 The module $\gr^\#\wM$ is a full lattice in an analogously defined graded $\gr\wA_K$-module $\gr^\#M_K$,  as discussed in \cite{PS9} (without the
 $K$ subscript), and there is an analogous inclusion
$\iota_K:\gr^\# M_K\to\gr \wN_K$ of graded $A_K$-modules, with $\Imm\iota_K$ the kernel of the natural map $\gr \wN_K\to\gr (\wN_K/\wM_K)$.

Let $d=d(\lambda)=\dim\,N_{K,\lambda}$. Thus,
$\wM$ is a direct sum of $d$ copies of $\wDelta(\lambda)$ by the construction of $\wM_K$.
Also, $\gr \wN_K$ must contain a direct sum of $d$ copies of the modules $\gr \wDelta_K(\lambda)$, viewed, for the moment, as ungraded $\gr \wA_K$-modules.
($\wN_K$ will also contain a graded version of the direct sum; see below.) Since $\wN_K/\wM_K$ has zero $\lambda$-weight space, this direct sum
must be contained in $\Imm\,\iota_K$. By dimension considerations, it must equal $\Imm\iota_K$. Taking gradings into account, we have a
decomposition
 $$\gr^\# \wM_K\cong\bigoplus_{i=1}^d\gr\wDelta_K(\lambda)(m_i)$$
for some non-negative integers $m_1\leq m_2\leq\cdots \leq m_d$. As usual,
the notation $\gr\wDelta_K(\lambda)(m)$ indicates that the usual grade $0$
head of $\gr\wDelta_K(\lambda)$ (viz., $\gr\wDelta_K(\lambda)(0)$) is given grade $m$ ($m\in\mathbb Z$).

Write $\wM_\lambda$ as a direct sum
$$\wM_\lambda=\bigoplus_{n\geq 0} \wM_{\lambda,n}$$
where $\wM_{\lambda,n}$ is an $\sO$-module chosen as a complement to $\wM_\lambda\cap
(\rad\wA_K)^{n+1}\wN_K$ in $\wM_\lambda\cap(\rad\wA_K)^n\wN_K$. The latter $\sO$-module is pure in
$\wM$ and is the full $\lambda$-weight space in $\wM\cap(\rad\wA_K)^n\wN_K$. The image $\gr^\#\wM_{n,\lambda}$ of
$\wM_{\lambda,n}$ in $(\gr^\#\wM)_n$ is the full $\lambda$-weight space $(\gr^\#\wM)_{n,\lambda}$. The image is $0$ unless $n=m_i$ for some $i$. In fact, the number
of $m_i$ equal to $n$ is the rank of $\wM_{\lambda,n}$. We have
$$\begin{cases}\wM=\wA\wM_\lambda=\bigoplus_{n\geq 0}\wA\wM_{\lambda,n}\\[2mm]
\wA\wM_{\lambda,n}\cong\wDelta(\lambda)\otimes_{\sO}\wM_{\lambda,n}.\end{cases}$$
Here, in the tensor product, $\wM_{\lambda,n}$ is regarded as an $\sO$-module only.

Recall that $m_1$ is the smallest integer $n$ with $\wM_{\lambda,n}\not=0$. Thus,
$$\begin{cases} \wM\subseteq(\rad\wA_K)^{m_1}\wN_K\\[2mm]
\wM\not\subseteq(\rad\wA_K)^{m_1+1}\wN_K.\end{cases}$$
The two displays above show that the two naturally isomorphic $\sO$-submodules $\wM_{\lambda,m_1}\subseteq\wN$ and $\gr^\#\wM_{\lambda,m_1}\subseteq(\gr\wN)_{m_1}\subseteq
\gr\wN$ are generated as $\sO$-modules by strongly primitive elements of $\wN$ and $\gr\wN$, respectively.
Put $\wR=\wA\wM_{\lambda,m_1}\cong\wDelta(\lambda)\otimes_\sO\wM_{\lambda,m_1}$, and $R_K=K\wR$. Form the graded
module $\gr^\#\wR\hookrightarrow\gr\wN$ as constructed for $\wM$, and note $(\gr^\#\wR)_\lambda\subseteq{\gr^\flat}\wN$ identifies naturally with $\wM_{\lambda,m_1}$.

We may identify $\gr^\#\wR$ with a submodule
of $\gr^\#\wM$. As such,
\begin{equation}\label{isoabove}\gr^\#\wR_{\lambda,m_1}=\gr^\#\wM_{\lambda,m_1}\cong \gr\wN_{\lambda,m_1}\end{equation}
 has rank $\dim M_{K,\lambda,m_1}$.
Consequently, since $\gr \wN_K$ has a graded submodule consisting of a direct sum of $\dim\,M_{K,\lambda,m_1}$-copies
of $\gr\wDelta_K(\lambda)(m_1)$, the image of $\gr^\# \wR_K$ must be that submodule. Thus, there is a graded $\gr\wA_K$-isomorphism
$\gr^\#\wR_K\cong\gr\wDelta_K(\lambda)(m_1)\otimes_K \wM_{K,\lambda,m_1}$.

The isomorphism (\ref{isoabove}) has implications for the filtration terms $\wR_K\cap(\rad\wA_K)^n\wN_K$, $n\geq 0$, whose successive quotients define the
grades of $\gr^\# R$. In fact, it is easy to see
\begin{equation}\label{displayed}
(\rad \wA_K)^s \wR_K=\wR_K\cap(\rad \wA_K)^{m_1+s}\wN_K
,\quad s\geq 0.\end{equation}
(The proof that each $\wR_K\cap(\rad \wA_K)^n\wN_K$ must be some $(\rad \wA_K)^s\wR_K$ is an easy downward induction on
$n$, starting the largest $n$ for which $\wR_K\cap(\rad \wA_K)^n\wN_K\not=0$. The precise indexing can then be made from the
fact that $\wR_K\cap(\rad \wA_K)^{m_1+1}\wN_K$ is the largest intersection not equal to $\wR_K$.)

Since the grade $n$ terms of $\gr^\#\wR$ are defined by successive quotients of the filtration terms $\wR\cap(\rad \wA_K)^n\wN_K$, (\ref{displayed}) gives
$$ \gr^\#\wR \cong (\gr\wR)(m_1)\cong (\gr\wDelta(\lambda))(m_1)\otimes_\sO\wM_{\lambda,m_1}$$
in $\gr\wA$-grmod.

To complete the proof by induction on $\dim
 \wN_K$, we next check that $\wN/\wR$ satisfies the hypotheses of the lemma. By construction,
$\wN/\wM$ has a filtration with sections $\wDelta(\nu)$, $\nu\in\Lambda\backslash\{\lambda\}$, while $\wM/\wR$ is visibly an ungraded direct
sum of copies of $\wDelta(\lambda)$. Similarly, $\gr(\wN_K/\wM_K)$ must, by construction, have a graded filtration with sections
modules $\gr\wDelta_K(\mu)(s)$, $\mu\in\Lambda\backslash\{\lambda\}$, $s\geq 0$. The kernel of the map
$\gr(\wN_K/\wR_K)\to\gr(\wN_K/\wM_K)$ is the cokernel of the map $\gr^\# \wR_K\to\gr^\# \wM_K$, which is, visibly, a direct sum
of graded $\gr \wA_K$-modules $\gr\wDelta_K(\lambda)(t)$, $t\geq 0$. Thus, $\wN/\wR$ satisfies the hypotheses of the lemma.

By induction, ${\gr^\flat} (\wN/\wR)$ has a graded filtration with sections various modules $\gr\wDelta(\nu)(s)$, $\nu\in\Lambda$, $s\in\mathbb N$, here, and in
the statement of the lemma.
Applying $K\otimes_\sO-$, these become sections of a graded $\gr\Delta_K$-filtration of $\gr \wN_K/\wR_K$.
It is easily checked that the hypotheses of Lemma \ref{preliminarylemma}  apply to $\wN$  and $\wR$, so that the map $\gr\wN
\to\gr\wN/\wR$  is surjective. The latter map has kernel ${\gr^\flat}\wN\cap\gr^\#\wR$, which contains $\gr \wM_{\lambda,m_1} $. Thus,
$${\gr^\flat}\wDelta(\lambda)(m_1)\otimes_{\sO}\wM_{\lambda,m_1}\subseteq{\gr^\flat}\wN\cap\gr^\#\wR
\subseteq\gr^\#\wR\cong\gr\wDelta(\lambda)(m_1)\otimes_\sO\wM_{\lambda,m_1},$$
noting the isomorphism ${\gr^\flat}\wDelta(\lambda)(m_1)\otimes_{\sO}\wM_{\lambda,m_1}\cong(\gr\wA)\gr\wM_{\lambda,m_1}$. The
 lemma now follows by Lemma \ref{preliminarylemma}, applied to $\wE={\gr^\flat}\wN\cap\gr^\#\wR$, and induction. \end{proof}

\begin{remark}\label{bigremark} (a) Character arguments in the (common) Grothendieck group for $\wA_K$-mod and $\gr\wA_K$-mod show that $[\wN_K:\Delta_K(\lambda)]=[\gr\wN_K:\gr\Delta_K(\lambda)]$ for $\wN$ as in
Lemma \ref{mainlemma}. Also, $[\wN_K:\Delta_K(\lambda)]$ is obviously the same as $[\wN:\wDelta(\lambda)]$ and $[\gr\wN_K:\gr\Delta_K(\lambda)]$ counts the number of (collective occurrences) of
prestandard modules $\wD(\lambda)$ in any prestandard filtration of $\gr^\flat\wN$, viewed as an
ungraded $\gr\wA$-module (a full lattice in $\gr\wN_K$).

(b) Assuming Hypothesis \ref{hypothesis1}, any PIM $\wP(\lambda)$, $\lambda\in\Lambda$, may be used for $\wN$ in Lemma \ref{mainlemma}. Here $\gr^\flat \wP(\lambda)=\gr\wP(\lambda)$ by Lemma \ref{projectivegradedlemma}. It has a
prestandard filtration by Lemma \ref{mainlemma}, and the proof of Lemma \ref{mainlemma} shows that the top term of the filtration may be taken to be $\gr^\flat\wDelta(\lambda)$. Also, the kernel of the map from
$\wP(\lambda)=\wP^\flat(\lambda)$ onto $\gr^\flat\Delta(\lambda)$ is filtered by prestandard modules $\wD(\mu)$, $\mu>\lambda$. This is also guaranteed by part (a) above of this remark.
\end{remark}

The following proposition is, in some sense, a corollary of the proof of Lemma \ref{mainlemma}, though additional
argument is required.

\begin{prop}\label{sense a corollary} Let $\wN$ be an $\wA$-lattice satisfying the hypothesis of
Lemma \ref{mainlemma} (which includes Hypotheses \ref{hypothesis1}). Assume also that $\gr\wDelta(\nu)$
has a simple head, for each $\nu\in\Lambda$ satisfying $[\wN_K:\wDelta_K(\nu)]\not=0$. Then
$\gr^\flat\wN=\gr\wN$ and all the prestandard modules in the filtration of Lemma \ref{mainlemma} are ``standard" (i.~e., if a shifted copy of $\wD(\nu)$ occurs in the filtration, then $\wD(\nu)\cong\gr\wDelta(\nu)$).
\end{prop}

Before giving the proof, we record the following immediate corollary (of the proposition and Lemma
\ref{mainlemma}, using the Morita equivalence of Proposition \ref{Morita2}).

\begin{cor}\label{Nameless-1}Suppose that $\wA$ is a split QHA over $\sO$ with poset $\Lambda$, and that $\gr\wA_K$ is a
QHA algebra with the same poset $\Lambda$. Let $\wN$ be a $\wA$-lattice  with a $\wDelta$-filtration, and suppose $\gr\wDelta(\nu)$ has a simple head for each $\nu\in\Lambda$ satisfying $[\wN_K:\wDelta_K(\nu)]\not=0$. Then $\gr\wN$ has a filtration with sections $\gr\wDelta(\lambda)$, $\lambda\in\Lambda$.
\end{cor}

\begin{proof}[Proof of Propostion \ref{sense a corollary}] According to Lemma
\ref{projectivegradedlemma}, the head of a given $\gr\wDelta(\lambda)$ guarantees that $\gr^\flat\wDelta(\lambda)=\gr\wDelta(\lambda)$. Since any
 module $\wD(\lambda)$ is sandwiched between $\gr^\flat\wDelta(\lambda)$ and $\gr\wDelta(\lambda)$, we must have $\wD(\lambda)=\gr\wDelta(\lambda)$.

We now continue with the notation and proof of Lemma \ref{mainlemma}. The displayed inclusions
in the last paragraph are now equalities, and so $\gr^\flat \wN\cap \gr^\#\wR=\gr^\#\wR$. That is,
$\gr^\#\wR\subseteq\gr^\flat\wN$.

We recall that $\gr^\#\wR$ is the kernel of the map $\gr\wN\to\gr\wN/\wR$. This map sends $\gr^\flat\wN$
onto $\gr^\flat\wN/\wR$, according to the paragraph quoted above. The paragraph before that, in the proof
of Lemma \ref{mainlemma}, notes that $\wN/\wR$ satisfies the hypotheses of the lemma required on
$\wN$. In particular, $\wN/\wR$  has a $\wDelta$-filtration. This is part of a $\wDelta$-filtration of $\wN$,
since $\wR$ is a direct sum of standard modules by construction (see its introduction earlier in the proof of
the lemma).

It follows that $\wN/\wR$ satisfies the hypothesis of the current proposition, and we can assume inductively
(using induction on the rank of $\wN$) that $\gr^\flat\wN/\wR=\gr\wN/\wR$. But now $\gr^\flat\wN$ and
$\gr\wN$ have the same image and kernel, so must be equal.
This proves the proposition. 
\end{proof}

As a further corollary of Lemma \ref{mainlemma}, Proposition \ref{sense a corollary}, and their proofs,
we have

\begin{cor}\label{Nameless} Let $\wA,\wN,\Lambda$ satisfy the hypotheses of the preceding corollary, and let $\Gamma\subseteq \Lambda$ be a poset ideal. Then $\gr\wN_\Gamma\cong(\gr\wN)_\Gamma$ as $\gr\wA$-lattices. In particular, the natural map $\gr\wN\to\gr\wN_\Gamma$ is surjective.\end{cor}

\begin{proof} Again, we may assume the hypothesis of Proposition \ref{sense a corollary}, using Morita
equivalence. Now, we simply continue the discussion, in the proof of the latter result,  choosing
$\lambda$ to belong to $\Gamma$. (If $\Gamma=\emptyset$, there is nothing to prove.) Proceeding by
induction on the rank of $\wN$, it can be assumed that
$$\gr(\wN/\wR)_\Gamma\cong(\gr\wN/\wR)_\Gamma.$$

Since $\wR$ is a direct sum of copies of $\wDelta(\lambda)$, it is contained in the kernel of the surjection
$\wN\to \wN_\Gamma$. Thus, $(\wN/\wR)_\Gamma\cong \wN_\Gamma$, and so $\gr(\wN/\wR)_\Gamma
\cong\gr\wN_\Gamma$. Also, since $\gr^\#\wR$ is a direct sum of shifted copies of $\gr\wDelta(\lambda)$, each of which has head which is a shift of $L(\lambda)$,  $\gr^\#\wR$ is contained in the kernel
of the surjection $\gr\wN\to(\gr\wN)_\Gamma$. Of course, $\gr^\#\wR$ is the kernel of the map
$\gr\wN\to\gr\wN/\wR$. The latter map is surjective, as Lemma \ref{mainlemma} and Proposition \ref{sense a corollary} (and their proofs) show. Thus, $(\gr\wN)_\Gamma\cong(\gr\wN/\wR)_\Gamma$. The corollary now
follows.\end{proof}

An analogue of Corollary \ref{Nameless} holds for QHAs over fields (without any integral hypothesis or
conclusion). See Appendix I.

We are now ready to establish a main theorem of this paper.

\begin{theorem}\label{mainIntGrthm} Suppose that $\wA$ is a split QHA over $\sO$ with poset $\Lambda$,
and $\gr\wA_K$ is a QHA over $K$ with the same poset $\Lambda$.
Further, suppose that $\gr\wDelta(\lambda)$ has a simple head,
for all $\lambda\in\Lambda$. Then $\gr\wA$ is a split QHA over $\sO$ with poset $\Lambda$ and standard objects $\gr\wDelta(\lambda)$,
$\lambda\in\Lambda$. \end{theorem}

\begin{proof} Using Proposition \ref{Morita2} to pass to a Morita equivalent algebra, we can assume that Hypothesis \ref{hypothesis1}
holds.  Let $\wP$ be any $\sO$-finite projective $\wA$-module. Then $\wN:=\wP$ satisfies the hypotheses of Proposition \ref{sense a corollary}.
 Now apply Remark \ref{bigremark}(b) to $\wP=\wP(\lambda)$, for any given $\lambda\in\Lambda$. This provides
an exact sequence $0\to\Omega\to\gr\wP(\lambda)\to\gr\wDelta(\lambda)\to 0$ in which $\Omega$ has a filtration with sections $\gr\wDelta(\nu)$, $\nu>\lambda$.
In particular,
$$
\left\{\!\!\!\begin{array}{rcl}
\Hom_{\gr\wA}(\Omega,L(\mu)) & = & 0,\\[2mm]
\Hom_{\gr\wA}(\Omega,\gr\wDelta(\mu)) & = & 0,
\end{array}\right.
$$
 unless $\mu>\lambda$. Thus
 \begin{equation}\label{bigExtvanishing}
\left\{\!\!\!\begin{array}{rcl}
\Ext^1_{\gr\wA}(\gr\wDelta(\lambda),L(\mu)) & = & 0,\\[2mm]
 \Ext^1_{\gr\wA}(\gr\wDelta(\lambda),\gr\wDelta(\mu)) & =& 0,
\end{array}\right.
 \end{equation}
  unless $\lambda<\mu$ ($\lambda,\mu\in\Lambda$).
In particular, if $\wM$ is a $\gr\wA$-module such that the composition factors  $L(\mu)$ of $\overline {\wM}$ satisfy $\nu\not>\lambda$, then
$\Ext^1_{\gr\wA}(\gr\wDelta(\lambda),\wM)=0$.\footnote{There are many ways to prove this. First, note the hypothesis on $\overline{\wM}=\wM/\pi \wM$ holds for any
of its homomorphic images $\pi^{a-1}\wM/\pi^a\wM$, with $a>0$, and, thus, on $\wM/\pi^a\wM$. The required vanishing holds if $\wM$ is replaced by $
\wM/\pi^a \wM$, so it
suffices to show it holds when $\wM$ is replaced by $\pi^a\wM$. For large enough $a$, the latter module is torsion free. So we may assume to start that $\wM$ is
torsion free. All composition factors $\wL_K(\mu)$ of $\wM_K$ satisfy $\mu\not >\lambda$ since the analogous property over $k$ holds. Thus, $\Ext^1_{\gr\wA_K}(\gr\Delta_K(\lambda),M_K)=0$, since
$\gr\wA_K$ is a QHA with poset $\Lambda$. Thus, $\Ext^1_{\gr\wA}(\gr\wDelta(\lambda),\wM)$ is torsion free. However, multiplication by $\pi$ on this $\Ext^1$-group
is surjective, since $\Ext^1_{\gr\wA}(\gr\wDelta(\lambda),\overline{\wM})=0$. So the $\Ext^1$-group itself vanishes.   }

  Next, take $\wP=\wA$, viewed as a left module
over itself in Proposition \ref{sense a corollary}. We can rearrange the (ungraded version of the) filtration guaranteed by
Proposition \ref{sense a corollary} so that all copies of a given $\gr\wDelta(\lambda)$
occur contiguously. Removing all redundant filtration indices $j$ and those for which $\wF_{j+1}/\wF_j$ and $\wF_j/\wF_{j-1}$ are isomorphic to the same
module $\gr\wDelta(\lambda)$, gives  a new filtration
$$0\subseteq\wF_{j_1}\subseteq\wF_{j_2}\subseteq\cdots\subseteq\wF_{j_s}=\gr\wA,
$$
with $1=j_1<j_2<\cdots<j_s=r$ and $\wF_{j_{i+1}}/\wF_{j_i}\cong\wDelta(\lambda_i)^{\oplus m_i}$, $i=1,\cdots, s$, for
suitable elements $\lambda_i\in\Lambda$ and positive integers $m_i$. We may assume
that $\lambda_i\not=\lambda_{i'}$ if $i\not= i'$, and also that $\lambda_i\not <\lambda_{i'}$ if $i<i'$.

Now take $\wJ_i=\wF_{j_i}$. By definition, $\wJ_i$ is a left ideal in $\gr\wA$. On the other hand, let $x\in\gr\wA$. We claim that  $\wJ_ix\subseteq \wJ_i$.
Right multiplication by $x$ defines an endomorphism $\gr\wA\to\gr\wA$, and it suffices to show that the induced map $\wJ_i\overset{x}\to\gr\wA\to(\gr\wA)/\wJ_i$ is
the zero map. If this map is nonzero, tensoring with $K$ defines a nonzero map $\wJ_{iK}\to (\gr\wA_K)/\wJ_{i K}$. However, $\wJ_{iK}$ is filtered
by standard modules $\gr\Delta_K(\lambda_j)$, $j=1,\cdots,s'$ for some non-negative integer $s'=s'_i$, while $(\gr\wA_K)/\wJ_{iK}$ is filtered by modules $\gr\Delta_K(\mu)$ with no
$\mu$ greater than, or equal to, any $\lambda_j$. Hence, $\Hom_{\gr\wA_K}(\wJ_{iK},(\gr\wA_K)/\wJ_{iK})=0$, a contradiction. Therefore, $\wJ_ix\subseteq \wJ_i$ as claimed, and $\wJ_i$ is an ideal in $\gr\wA$.

Next, for $\lambda\in\Lambda$, $\End_{\gr\wA}(\gr\wDelta(\lambda))$ is a finite $\sO$-subalgebra of
$$\End_{\gr \wA}(\gr\wDelta(\lambda))_K=\End_{\gr\wA_K}(\gr\Delta_K(\lambda))\cong K,$$
so $\End_{\gr\wA}(\wDelta(\lambda))\cong\sO$. Thus,
 $$\End_{\gr\wA}(\wJ_i/\wJ_{i-1})\cong\End_{\gr\wA}(\wDelta(\lambda_i)^{\oplus m_i})\cong M_{m_i}(\sO).$$

Fix $i$, and put $\Gamma_i=\{\gamma\in\Lambda\,|\,\gamma\leq\lambda_{i'}, {\text{\rm for some}}\,i'>i\}$. Then $\wJ_{ik}\subseteq B:=(\gr\wA)_k$ and $B/\wJ_{ik}$ is the largest quotient algebra of $B$ whose composition
factors $L(\nu)$ satisfy $\nu\in\Gamma_i$. Since $\wJ_{ik}/\wJ_{ik}^2$ is a $B/\wJ_{ik}$-module, its composition factors $L(\gamma)$ satisfy $\gamma\in\Gamma_i$.
On the other hand,  $\wJ_{ik}/\wJ_{ik}^2$ is a homomorphic image of $\wJ_{ik}$, which is filtered with sections $\gr\wDelta(\lambda)$, $\lambda\not\in\Gamma_i$, each of which has head $L(\lambda)$. So, if
$\wJ_{ik}/\wJ_{ik}^2\not=0$, it must have a composition factor $L(\lambda)$, $\lambda\not\in\Gamma_i$. Therefore, $\wJ_{ik}^2=\wJ_{ik}$, so $\wJ_i=\wJ_i^2$ by
Nakayama's lemma.

Finally, to show $\wJ_i/\wJ_{i-1}$ is $(\gr\wA)/\wJ_{i-1}$-projective, it suffices to show that $\gr\wDelta(\lambda_i)$ is $(\gr\wA)/\wJ_{i-1}$-projective. But this
follows from the sentence immediately following (\ref{bigExtvanishing}).

 In conclusion, $\gr\wA$ is a split
QHA over $\sO$, as claimed. This construction also shows that $\gr\wDelta(\lambda)$ is standard.
\end{proof}

 \begin{remark}\label{bigremark2} (a) We could give an alternative version of part of the above proof
 by appealing to Corollary \ref{Nameless}. In particular, we note, for $\wA$ and $\Lambda$ satisfying the hypothesis of Theorem \ref{mainIntGrthm}, and any poset ideal $\Gamma\subseteq\Lambda$, that
 $$ \gr(\wA_\Gamma)\cong (\gr \wA)_\Gamma.$$

 (b) Consequently, the map $\gr\wA\to\gr\wA_\Gamma$ is surjective. So, if $\wN$ is an $\wA_\Gamma$-lattice,
 with $\wA$ and $\Gamma$ as above, there is a physical equality of $\sO$-modules
 $$\gr^\flat_\wA\wN=\gr^\flat_{ \wA_\Gamma}\wN.$$

 (c) In a similar spirit, another consequence of Corollary \ref{Nameless}: If $\wA,\Lambda,\wN$ satisfy the
 hypothesis of the corollary, {\it then the strongly primitive elements of $\gr\wN$ are the same as the primitive elements of $\gr\wN$.} This follows by taking $\Gamma=\{\gamma\,|\,\gamma\not>\lambda\}$ for
 any given $\lambda\in\Lambda$, and using the isomorphism
 $$\gr\wN/\wN\cap \wN'_K(\lambda)=\gr\wN_\Gamma\cong(\gr\wN)_\Gamma
\cong\gr \wN/\gr\wN\cap (\gr\wN)'_K(\lambda).$$
 \end{remark}

\section{A special case}  In this section, unless otherwise noted, we continue the notation and assume both the conditions of Hypotheses \ref{hypothesis1}
of the previous section. In particular, the algebra $\wA$ is a split QHA
with standard objects $\wDelta(\lambda)$, $\lambda\in\Lambda$. The irreducible modules  and PIMS of the QHA $\wA_K$ are denoted $L_K(\lambda)$ and $P_K(\lambda)$, respectively. We sometimes write
$\Delta_K(\lambda)$ (which is isomorphic to $\wDelta(\lambda)_K$) for its standard modules. The
algebra $\gr \wA_K$
 is also a QHA under Hypothesis \ref{hypothesis1}, specifically item (2). It has standard objects $(\gr\wDelta(\lambda))_K\cong
\gr\wDelta(\lambda)_K=\gr\Delta_K(\lambda)$, $\lambda\in\Lambda$.

In addition, we assume, for the rest of this section,  that $\wA$ has a pure subalgebra $\widetilde\fa$ and a Wedderburn complement $\wA_{K,0}$ of
$\wA_K$. For use in the results below, we record the following conditons.

\medskip
\begin{conditions}\label{conditions}   \begin{itemize}
\item[(1)] $\wfa_K$ has a tight grading $\wfa_K=\wfa_{K,0}\oplus\wfa_{K,1}\oplus\cdots $\footnote{This means that $\wfa_K$ is positively graded,
$\wfa_K=\bigoplus_{n\geq 0}\wfa_{K,n}$ with $\wfa_{K,0}$ semisimple and with $\wfa_K$ generated by
$\wfa_{K,1}$; see \cite[\S4]{CPS1a}. Equivalently, $\wfa_K\cong \gr\wfa_K$.}.

\item[(2)] $
\rad\wA_K=(\rad\widetilde\fa_K)\wA_K=\wA_K(\rad\widetilde\fa_K).$

\item[(3)]
For $\lambda\in\Lambda$, $\Delta_K(\lambda)$ has a graded $\widetilde\fa_K$-structure, and is generated
as an $\wfa_K$-module by
 $\Delta_K(\lambda)_0$.

\item[(4)] In (3),  $\Delta_K(\lambda)_0$ is $\wA_{K,0}$-stable. Also, $\wA_{K,0}$ contains  $\wfa_{K,0}$ (defined in (1)) and all idempotents $e_\lambda, \lambda\in\Lambda$.

\item[(5)] $\wfa$ has a positive grading $\wfa=\oplus_{r\geq 0}\wfa_r$  such that $K{\wfa}_r=\wfa_{K,r}$, the $r$th grade of $\wfa_K$ in (1), for each $r\in\mathbb N$.
\end{itemize}
\end{conditions}

Observe that Conditions (\ref{conditions})(1) \& (5) can be made independently of the QHA algebra
$\wA$. In this spirit, we have the following result.

\begin{prop}\label{three part prop} Let $\wfa$ be an algebra which is free and finite over the DVR $\sO$ and
satisfies  Conditions  \ref{conditions}(1) \& (5). Then the following statements hold:

(a) For each $r\in\mathbb N$, we have $\wfa_r=\wfa\cap\wfa_{K,r}$ and
$\sum_{i\geq r}\widetilde{\fa}_i=\widetilde\fa\cap\rad^r\widetilde\fa_K$. There is an isomorphism
$\widetilde\fa\overset\sim\to\gr\widetilde\fa$ of graded $\sO$-algebras sending $\x\in\widetilde{\fa}_r$ to
its image $[x]$ in $\widetilde\rad^r\wfa/\widetilde\rad^{r+1}\wfa =(\gr\widetilde{\fa})_r$.

(b) If $\wM$ is an $\widetilde\fa$-lattice, the following statements are equivalent:

\begin{itemize}

\item [(i)] $\wM$ is tight. (See Definition \ref{tightdef1}.)

\item[ (ii)] $\sum_{i\geq r}{\widetilde\fa}_i\wM=\widetilde{\rad}^r\wM$ for each $r\in\mathbb N$.

\item[ (iii)] The $\gr\widetilde\fa$-lattice $\gr\wM$ is generated by $(\gr\wM)_0$.
\end{itemize}
\end{prop}

\begin{proof} We begin with (a). Since $\widetilde\fa_r$ is an $\sO$-direct summand of $\widetilde\fa$, we have that
$\wfa_r=\wfa\cap\wfa_{r,K}=\wfa\cap\wfa_{K,r}$. Similarly, we get $\sum_{i\geq r}\wfa_i=\wfa\cap\rad^r\wfa_K$, since
$$ \rad^r\wfa_K=(\rad\wfa_K)^r=(\sum_{i\geq 1}\wfa_{K,i})^r=\sum_{i\geq r}\wfa_{K,i}=
(\sum_{i\geq r}\wfa_i)_K.$$

In particular, there is a well-defined $\sO$-linear graded map sending $x\in\wfa_r$ to its image $[x]\in
\wfa\cap\rad^r\wfa_K/\wfa\cap\rad^{r+1}\wfa_K$, and the map is clearly multiplicative. Thus, it is a graded homomorphism of $\sO$-algebras. The equations $\sum_{i\geq r}\wfa_i=\wfa\cap\rad^r\wfa_K$ and $\sum_{i\geq r+1}\wfa_i=\wfa\cap\rad^{r+1}\wfa_K$ show that the map is an isomorphism. Part (a) is now proved.

We now prove (b). The equivalence (i) $\iff$ (ii) is immediate from the sum formula for $\wfa\cap\rad^r\wfa_K$ in part (a).

Now suppose that (ii) holds. Then, using the isomorphism of (a),
$$(\gr\wfa)_r(\gr\wM)_0=[\wfa_r]\wM/(\wM\cap\rad\wM_K)=(\wfa_r\wM+ \wM\cap\rad^{r+1}\wM_K)/\wM\cap\rad^{r+i}\wM_K,$$
the second equality holding by definition of the action of $\gr\wfa$ on $\gr\wM$.  But (ii) gives that
$$\wfa_r\wM+\wM\cap\rad^{r+1}\wM_K=\wfa_r\wM +\sum_{i\geq r+1}\wfa_i\wM=\sum_{i\geq r}\wfa_i\wM=\wM\cap
\rad^r\wM_K.$$
Thus, $(\gr\wfa)_r(\gr\wM)_0=\gr\wM_r$, and (iii) holds.

If (iii) holds, then $(\gr\wfa)_r(\gr\wM)_0=\gr\wM_r$ for all $r\in\mathbb N$. This gives that
$$\wfa_r\wM+\wM\cap\rad^{r+1}\wM_K=\wM\cap\rad^r\wM_K$$
calculating with $(\gr\wfa)_r=[\wfa_r]$, as above. Downward induction on $r$ proves (ii). (Note that (ii) holds for
sufficiently large $r$, with both sides equal to 0.)  The proposition is proved.
\end{proof}

For $\lambda\in\Lambda$, define the $\wA_K$-submodule $E_K(\lambda)$ of $P_K(\lambda)$   by the
short exact sequence
\begin{equation}\label{sex} 0\to E_K(\lambda)\hookrightarrow P_K(\lambda)\to \Delta_K(\lambda)\to 0.
\end{equation}

\begin{theorem}\label{bigtheorem} Let $\wA$ be a QHA over $\sO$ with weight poset $\Lambda$, which satisfies Hypothesis \ref{hypothesis1}
and Conditions \ref{conditions}(1)---(5).
Let $\lambda\in\Lambda$ be such that the PIM $P_K(\lambda)$ in $\wA_K$-mod contains a full $\wA$-lattice $\wP(\lambda)^\dagger$ with the following properties:

(i) $\wP(\lambda)^\dagger=\wA v$, where $v\in \wP(\lambda)^\dagger_\lambda$;

(ii) $\wP(\lambda)^\dagger=\wP(\lambda)^\dagger_0\oplus \wP(\lambda)^\dagger\cap\rad P_K(\lambda)$, where $\wP(\lambda)^\dagger_0$ is an $\sO$-direct summand of $\wP(\lambda)^\dagger$ such that $K\wP(\lambda)^\dagger_0 + E_K(\lambda)$ is an $\wA_{K,0}$-submodule of
$K\wP(\lambda)^\dagger=P_K(\lambda)$;

(iii) $\wP(\lambda)^\dagger|_\wfa$ is tight.

Then $\wDelta(\lambda)|_\wfa$ is tight.
\end{theorem}

\begin{proof}Since $L_K(\lambda)\cong P_K(\lambda)/\rad P_K(\lambda)$, (ii) implies  $K\wP(\lambda)^\dagger_0+ E_K(\lambda)/E_K(\lambda)\cong
L_K(\lambda)|_{A_{K,0}}$ as an $A_{K,0}$-module. Now let $\phi:P_K(\lambda)\to\Delta_K(\lambda)$ be a surjection. Then $\phi(\rad P_K(\lambda))=\rad\Delta_K(\lambda)$, so that $\phi(K\wP(\lambda)^\dagger_0)$ is nonzero and isomorphic to $L_K(\lambda)|_{A_{K,0}}$. Since the latter module appears with multiplicity one as an $\wA_{K,0}$-composition factor of $\Delta_K(\lambda)$, we must have $\phi(K\wP(\lambda)^\dagger)=\Delta_K(\lambda)_0$, the $A_{K,0}$-submodule of $\Delta_K(\lambda)$ in Conditions (\ref{conditions})(3) \& (4). (Observe that
$$\Delta_K(\lambda)_0\cap\rad\Delta_K(\lambda)=\Delta_K(\lambda)_0\cap\sum_{i\geq 1}\wfa_{K,i}\Delta_K(\lambda)_0=0$$
by Conditions (\ref{conditions})(1), (2), (3).)

Since $\wP(\lambda)^\dagger_0$ is a complement in $\wP(\lambda)^\dagger$ to $\wP(\lambda)^\dagger\cap\rad P_K(\lambda)=\wP(\lambda)^\dagger\cap\rad\wP(\lambda)^\dagger_K$,  which is $\sum_{i\geq 1}\wfa_i\wP(\lambda)^\dagger$, by the assumed tightness (iii) and
Proposition \ref{three part prop}, we have
$$\wP(\lambda)^\dagger=\wP(\lambda)_0+\sum_{i\geq 1}\wfa_i\wP(\lambda)=\wP(\lambda)^\dagger_0+\sum_{i\geq 1}\wfa_i\wP(\lambda)^\dagger_0+\sum_{i\geq 1}\wfa_i\wP(\lambda)^\dagger).$$
Further iterating this equation (or just applying Nakayama's lemma), we get that
$$\wP(\lambda)^\dagger=\wP(\lambda)^\dagger_0+\sum_{i\geq 1}\wfa_i\wP(\lambda)^\dagger_0.$$
Thus,
$$\phi(\wP(\lambda)^\dagger)=\phi(\wP(\lambda)^\dagger_0)+\sum_{i\geq 1}\wfa_i\phi(\wP(\lambda)^\dagger_0).$$
Since $K\phi(\wP(\lambda)^\dagger_0)=\phi(K\wP(\lambda)^\dagger_0)=\Delta_K(\lambda)_0$, the $\wfa_K$-graded structure of Condition (3) shows the above sum is direct, even inside the right-hand side sum. That is,
$$\sum_{i\geq 1}\wfa_i
\phi(\wP(\lambda)^\dagger_0)=\bigoplus_{i\geq 1}\wfa_i\phi(\wP(\lambda)_0^\dagger).$$

We claim that the $\wfa$-lattice $\phi(\wP(\lambda)^\dagger)|_\wfa$ is tight. By Propositon \ref{three part prop}, it is sufficient to check the equality
$$\sum_{i\geq r}\wfa_i\phi(\wP(\lambda)^\dagger)=\phi(\wP(\lambda)^\dagger)\cap\rad^r\phi(\wP(\lambda))^\dagger_K, \quad\forall r\in\mathbb N.$$
For $r=0$, this is trivial, since $\sum_{i\geq 0}\wfa_i=\wfa$, and $\wfa\phi(\wP(\lambda)^\dagger)=\phi(\wP(\lambda)^\dagger)$.
For $r\geq 1$, we may argue as above that $\sum_{i\geq r}\wfa_i\wP(\lambda)^\dagger=\sum_{i\geq r}\wfa_i\wP(\lambda)_0.$ Consequently,
$$\sum_{i\geq r}\wfa_i\phi(\wP(\lambda)^\dagger)=\sum_{i\geq r}\wfa_i\phi(\wP(\lambda)^\dagger_0).$$
Multiplying by $K$, the left-hand side gives $(\rad\wfa_K)^r\phi(K\wP(\lambda)^\dagger)=\rad^r\phi(\wP(\lambda)^\dagger)_K$.
The right-hand side, however, is an $\sO$-direct summand of $\phi(\wP(\lambda)^\dagger)$, as shown by the direct sum decomposiiton above. Hence, it is equal to the intersection with $\phi(\wP(\lambda)^\dagger)$ of its product
with $K$, i.~e., it is equal to
$$\phi(\wP(\lambda)^\dagger)\cap\rad^r\phi(\wP(\lambda)^\dagger_K),$$
and so $\phi(\wP(\lambda)^\dagger)$ is tight as an $\wfa$-lattice. Since $\phi(\wP(\lambda)^\dagger)=\wA\phi(v)$, with $v$ and
$\phi(v)$ being $\lambda$-weight vectors, we have $\phi(\wP(\lambda)^\dagger)\cong\wDelta(\lambda)$, proving
the theorem.
\end{proof}

\begin{cor}\label{cor4.4} Let $\wA$ be a QHA over $\sO$ with weight poset $\Lambda$. If $\lambda\in\Lambda$ satisfies the hypothesis of Theorem \ref{bigtheorem}, then $\gr\wDelta(\lambda)$ has a simple head. Explicitly, $\text{\rm head}(\gr\wDelta(\lambda))\cong
L(\lambda)$, the head of $\wDelta(\lambda)$. The assertion still holds if $\lambda$ belongs to an
ideal $\Gamma$, and $\gr\wDelta(\lambda)$ is regarded as a $\gr \wA_\Gamma$-module.\end{cor}

\begin{proof}
By Theorem \ref{three part prop}, $\wDelta(\lambda)$ is tight as a $\wfa$-lattice. Next, note the action of $\gr\wfa$ on $\gr\wDelta(\lambda)$ factors naturally through the graded map $\gr\wfa\to\gr\wA$. By Propostion \ref{three part prop}(c), $(\gr\wDelta(\lambda)_0$ generates $\gr\wDelta(\lambda)$ as a $\gr\wfa$-module, and thus as a $\gr\wA$-module.  But
$(\gr\wDelta(\lambda))_0$ has a simple $\gr\wA$-head, since it has a simple $\wA$-head (which, of course, is also the head $L(\lambda)$ of $\wDelta(\lambda)$).
\end{proof}

Now we can use Theorem \ref{mainIntGrthm}, applied to ideals in $\Lambda$, to obtain some integral QHAs
of the form $\gr\wA_\Gamma$.   The assertion below regarding $\wN$ follows from Corollary \ref{Nameless-1}.

\begin{theorem}\label{integralquasihereditary} Let $\wA$ be a QHA over $\sO$ with weight poset $\Lambda$.  If $\Gamma$ is an ideal in $\Lambda$ such that the hypothesis of Theorem \ref{bigtheorem} holds for
all $\lambda\in\Gamma$, then
  $\gr\wA_\Gamma$ is a QHA (over $\sO$) with weight poset $\Gamma$, and
standard modules $\gr\wDelta(\lambda)$, $\lambda\in\Gamma$.
 In addition,  if $\wN$ is a $\wA_\Gamma$-lattice with a $\wDelta$-filtration, then the $\gr\wA$-lattice $\gr\wN$ has a
$\gr\wDelta$-filtration  (a filtration with sections $\gr\wDelta(\mu)$, $\mu\in\Gamma$. \end{theorem}

\section{Quantum enveloping algebras}
For the rest of this paper, let $\Phi$ be a (classical) irreducible root system. Let $C=(c_{i,j})$ be the $r\times r$ Cartan matrix of $\Phi$; thus, $c_{i,j}=(\alpha_i^\vee,\alpha_j)$ if $\Pi=\{\alpha_1,\cdots,\alpha_r\}$ is a fixed simple set of roots of $\Phi$.  Let $\Phi^+$ be the corresponding set of positive roots, and let $\alpha_0\in\Phi^+$ be the
maximal short root. For $\alpha\in\Phi$, let
\begin{equation}\label{dalpha}
d_\alpha:=(\alpha,\alpha)/(\alpha_0,\alpha_0)\in\{1,2,3\}.\end{equation}

 Next, let  $p>2$ be a prime integer.
If $\Phi$ has type $G_2$, then assume $p>3$. (Later additional assumptions will be made on $p$.)
Thus, $p$ does not divide any of the integers
$d_\alpha$ defined above in (\ref{dalpha}). In the polynomial ring
${\mathbb Z}[v]$, let ${\mathfrak m}=(p, v-1)$ and put $\sA={\mathbb Z}[v]_{\mathfrak m}$.  Setting
$\sO=\sA/(\phi_p(v))$, with $\phi_p(v)=v^{p-1} +\cdots +1$, the $p$th cyclotomic polynomial,
$(K,\sO,{\mathbb F}_p)$ is a $p$-modular system, where $K={\mathbb Q}(\zeta)$. Here $\zeta: = v
+ (\phi_p(v))$
is a primitive $p$th root of unity. Under the quotient map $\pi:\sO\to{\mathbb F}_p$, we have $\pi(\zeta)=1$.
Later it will be important to observe that, given $\alpha\in\Phi$,
\begin{equation}\label{units}
\zeta^{d_\alpha} - 1 = u_\alpha (\zeta -1),\quad{\text{\rm for some unit $u_\alpha\in\sO$.}}
\end{equation}
In fact, we can take $u_\alpha= \zeta^{d_\alpha-1} + \cdots+\zeta +1\in\sO$, which is a unit
since $\pi(u_\alpha)=d_\alpha\cdot 1$ in ${\mathbb F}_p$ by comments just above.

  Let $U'$ be the quantum enveloping algebra over the function field ${\mathbb Q}(v)$, with
 generators $E_\alpha,F_\alpha, K_\alpha, K_\alpha^{-1}$, $\alpha\in\Pi$,  satisfying the usual relations. For
 $\beta\in\Phi$, write $\beta=\sum_{\alpha\in\Pi} n_\alpha\alpha$ (for integers $n_\alpha$) and
 set
 \begin{equation}\label{Kbeta}K_\beta:=\prod_{\alpha\in\Pi} K_\alpha^{n_\alpha}\in U'.\end{equation}
  Let
$U'_\sA$ be the $\sA$-subalgebra generated by the divided powers $E^{(i)}_\alpha=E_\alpha^i/[i]_\alpha^!$,
$F^{(i)}_\alpha=F_\alpha^i/[i]_\alpha^!$ and the elements $K_\alpha$ $K_\alpha^{-1}$, $\alpha\in\Pi$, $i\in\mathbb N$.\footnote{Here $[i]^!_\alpha$ means that the polynomial $[i]^!$ is to be evaluated at $v^{d_\alpha}$.} Set
$\widetilde U'_\zeta:=\sO\otimes_\sA U'_\sA$ and $U'_\zeta=K\otimes_\sO \widetilde U'_\zeta$.

 The category $U'_\zeta$-mod of finite dimensional integrable, type 1 $U'_\zeta$-modules has irreducible modules
$L_\zeta(\lambda)$ indexed by the poset $X^+$ of dominant weights \cite[\S5.1]{Jan2}. In fact, it is a highest weight
category (in the sense of \cite{CPS-1}) with standard (resp., costandard) modules denoted $\Delta_\zeta(\lambda)$ (resp., $\nabla_\zeta(\lambda)$), $\lambda\in X^+$. Given a finite ideal $\Lambda$ in the poset $X^+$, let
$U'_\zeta$-mod$[\Lambda]$ be the full subcategory of $U'_\zeta$-mod consisting of objects whose composition factors have the form  $L_\zeta(\lambda)$, $\lambda
\in\Lambda$.  Let $\widetilde I$ be the annihilator in $\widetilde U'_\zeta$ of $U'_\zeta$-mod$[\Lambda]$, and
set $\wA=\wU'_{\zeta,\Lambda} :=\wU'_\zeta/\widetilde I$. Then $\wA$ is an integral QHA in the sense above.
See \cite{CPS7}, \cite{DS} for more details.

We can repeat the previous paragraph, replacing $\Lambda$ by a finite non-empty ideal of the poset
of
$p$-regular
dominant weights.\footnote{A weight $\lambda\in X$ is $p$-regular if $(\lambda+\rho,\alpha^\vee)\not\equiv 0$ mod $p$ for all roots $\alpha$. (Here $\rho$ is the Weyl weight.)  So we assume that (1) $\Lambda\subseteq X^+$ consists of $p$-regular weights, and (2) if $\mu\in X^+$ is $p$-regular 
and $\mu\leq \lambda$, then $\mu\in\Lambda$.}   We wish to verify that the Hypothesis \ref{hypothesis1} holds for the algebra $\wA=\wA_\Lambda$ for $\Lambda$ a suitably large ideal of $p$-regular weights. Recall from \cite[ \S8 ]{PS9} that  $\Lambda$ is {\it fat} provided, given any $p$-restricted, $p$-regular weight $\lambda$, we have $2(p-1)\rho+w_0\lambda\in\Lambda$.  This implies, in particular,  that $\Lambda$ contains all
$p$-restricted, $p$-regular dominant weights.

For $\alpha\in\Pi$, $K_\alpha^p-1$ is central in $U'_\zeta$. Consider  the $\sO$-Hopf algebra
$$\wU_\zeta:= \wU'_\zeta/\langle K_\alpha^p-1\,|\, \alpha\in\Pi\rangle.$$
Let $U_\zeta:=\wU_{\zeta,K}$. (Of couse, the $K_\alpha^p-1\in \widetilde I$.) Following Lusztig  \cite[6.5(b)]{L2} and \cite[Thm.~8.3.4]{L1}, let $\widetilde u_\zeta$ and $u_\zeta=\widetilde u_{\zeta,K}$
be the small quantum enveloping algebras.   Thus, $\widetilde u_\zeta$ (resp., $u_\zeta$) is a (normal) subalgebra of $\wU_\zeta$ (resp., $U_\zeta$). In addition, $\widetilde u_\zeta$ is a free $\sO$-module
of rank $p^{\dim {\mathfrak g}}$. Also, $u_\zeta$ admits a
triangular decomposition
$$u_\zeta^-\otimes u_\zeta^0\otimes u_\zeta^+\overset {\text{\rm mult}}\longrightarrow u_\zeta,$$
where $u_\zeta^+, u_\zeta^-, u_\zeta^0$ are certain subalgebras of the algebras $U^+_\zeta, U^-_\zeta, U^0_\zeta$ defined in \cite{L1}, \cite{L2}. 

  Finally, let $u_\zeta'$ is defined to be the product of the $p$-regular blocks in the algebra $u_\zeta$.
  Define $\widetilde u'_\zeta=\widetilde u_\zeta\cap u'_\zeta$. It is easy to see that $\widetilde u_\zeta'$ is a direct factor of $\widetilde u_\zeta$, and $k\otimes \widetilde u'_\zeta$ is the product of all
  regular blocks of $k\otimes\widetilde u_\zeta$.\footnote{Write $u_\zeta=u'_\zeta\oplus u^s_\zeta$, where
  $u^s_\zeta$ is the product of all singular blocks of $u_\zeta$. Then all composition factors of any ``reduction
  mod $p$" of (a full $\widetilde u_\zeta$-lattice in) a $u'_\zeta$-module, such as $k\otimes (\widetilde u_\zeta\cap u'_\zeta)$, have $p$-regular weights. A similar statement holds for $u^s_\zeta$, using $p$-singular weights. Now $\widetilde u_\zeta/(\widetilde u_\zeta\cap u'_\zeta \oplus \widetilde u_\zeta\cap u_\zeta^s)$ is a homomorphic
image of the lattices $\widetilde u_\zeta/ \widetilde u_\zeta\cap u'_\zeta$ and $\widetilde u_\zeta/\widetilde u_\zeta\cap u^s_\zeta$ in $u^s_\zeta$ and $u'_\zeta$, respectively. This common quotient must be zero, giving
  the
  claimed properties of $\widetilde u'_\zeta.$} 

 \begin{theorem}\label{preliminaryresult}  Let $\wA=\wU_{\zeta,\Lambda}$ for a finite, non-empty ideal $\Lambda$ of $p$-regular dominant weights.

(a) Hypothesis \ref{hypothesis1} holds provided $p>h$.

(b) If, in addition, $\Lambda$ is a fat ideal, then the image
$\wfa$ of $\widetilde u_\zeta$ in  $\wA$
is a pure subalgebra such that Conditions (\ref{conditions})(1)--(4) hold.

(c) Also, for
$p>h$ and $\Lambda$ fat, Condition \ref{conditions}(5) holds for  the grading in Condition \ref{conditions}(1).
\end{theorem}

\begin{proof} We begin by showing that Hypothesis \ref{hypothesis1}(1) holds. We have already remarked that $\wA$ is
an integral QHA. We need to check that $\wA$ contains an appropriate set of idempotents. We recall
from \cite[Cor.~3.3]{DS} that the left $\wU_\zeta$-module $\wA$ is a direct sum of weight spaces, i.~e., , spaces upon which $\wU_\zeta^0$ acts
in a uniform way determined by $\mu$.  Let $X=X_\Lambda$ be the
(integral) weights $\mu$ for which $\wU_\zeta$-module $\wA$ has a nonzero $\mu$-weight space. All weights
in $X$ are $p$-regular, and $\Lambda\subseteq X$. Let $\widetilde C$ be the image in $\wA$ of $\wU^0_\zeta$. Thus, $\wC$ is commutative. Since $\wA$ is a direct sum of rank 1 $\wC$-modules (on which $\wU_\zeta^0$ acts with some weight $\lambda\in X$), the algebra $\wC_K$ has a faithful completely reducible module,
namely, $\wA_K$.  Hence, $\wC_K$ is (split) semisimple and a direct sum of copies of $K$, one copy for each
weight. Similarly, the image $C$ of $\wC$ in $\wA_k=A$ is split semisimple, and its irreducible (1-dimensional)
modules are also indexed by the set $X$. If we base change $\wA$ to $\wA_{\widehat\sO}$, where $\widehat\sO$ is
the completion of $\sO$, idempotents in $\overline\wC\subseteq A$ lift to idempotents in
$\widehat\sO\wC\subseteq \wA_{\widehat\sO}$, preserving orthogonality. In this way, we get a set
of $|X|$ orthogonal primitive idempotents in $\widehat\sO\sC\subseteq K\widehat\sO\wC$. But $K\wC$
already has a complete set of $|X|$ orthogonal primitive idempotents in the (commutative)
algebra $K\wC\subseteq K\widehat\sO\wC$. By the uniqueness of complete sets of central primitive idempotents, the two sets
of $|X|$ idempotents we have constructed are the same. Hence, the idempotents $\{e_\mu\,|\,\mu\in X\}$
constructed for $K\wC$ actually
lie in $K\wC\cap\widehat\sO\wC=\wC\subseteq\wA$. This set of idempotents, together with its subset indexed by $\Lambda$, gives $\wA$ the
structure of a TQHA. Thus, Hypothesis \ref{hypothesis1}(1) holds.

Property (2) holds for $p>h$ by \cite[Thm.~8.4]{PS9} (with $p$ playing the role of $e$). This completes the proof of (a).

Now we prove (b).  Since $\Lambda$ is fat, $
\widetilde u'_{\zeta,K}=u'_\zeta$ maps injectively onto its image $\wfa_K$ in
$\wA_K$. The same holds for $\widetilde u'_\zeta$ since $\widetilde u'_\zeta$ is contained in $u'_\zeta$, so, in particular, $\widetilde u'_\zeta\cong \wfa$.  Similarly, using fatness, $k\otimes\widetilde u'_\zeta$ maps isomorphically onto the
image of $k\otimes\wfa$ in
$k\otimes \wA$. The latter image is not {\it a priori} isomorphic to $k\otimes\wfa$, but follows here, because
$\widetilde u'_\zeta\cong\wfa$. Now Lemma \ref{wellknownlemma} implies that $\wfa$ is pure in $\wA$.

 In \cite[\S\S 18.17--18.21]{AJS}, it is proved that $u'_\zeta$ is a Koszul algebra, a property which implies that
 $u'_\zeta$ has a tight grading (i.~e., a grading making it isomorphic to $\gr u'_\zeta$), provided that
 $p>h$ and the Lusztig character conjecture holds in the quantum enveloping algebra case (which is
 true \cite{T}). Thus,
Condition \ref{conditions}(1) holds. Condition \ref{conditions}(2) follows from left-right symmetry and \cite[Lem.~8.3]{PS9}, and
Condition \ref{conditions}(3) follows from \cite[Thm.~6.4]{PS9}. In fact, choose a larger ideal $\Lambda'$
containing $\Lambda$ and consisting of $p$-regular weights, such that the projective cover $P(\lambda)$
of $L(\lambda)$ in $U_{\Lambda'}$ is projective for $\wfa'_K$. The final condition of (4) follows from
Lemma \ref{preliminarysectionlemma}.  This proves (b).

The proof of (c) requires further results from \cite{AJS}, and it is given in \S7, Appendix II. \end{proof}

\begin{remark} As remarked in the  proof, the graded algebra $u'_\zeta$ is isomorphic to 
$\gr u'_\zeta$ as graded algebras. Also, Conditions \ref{conditions}(2)  show that the 
$\gr \wfa_K\subseteq\gr\wA_K$ (whether $\Lambda$ is fat or not). Thus, the
natural map $\widetilde u'_\zeta\to\wfa$ induces a graded map 
$\widetilde u'_\zeta\to\gr \wA$, under the hypotheses of the theorem (using the grading on 
$\widetilde u'_\zeta$ in the proof of part (c)).
The map is an injection when $\Lambda$ is fat.
\end{remark}
\begin{theorem}\label{mainquantumresult} Assume that $p>2h-2$ is an odd prime and consider the algebra $\wA=\wU_{\zeta,\Gamma}$, where $\Gamma$
is an non-empty ideal of $p$-regular dominant weights. Then $\gr\wA$ is a QHA over $\sO$, with
standard modules $\gr\wDelta(\lambda)$, $\lambda\in\Gamma$.

\end{theorem}

\begin{proof}
Our goal is to verify the hypotheses of Theorem \ref{integralquasihereditary} (which are the hypotheses of Theorem \ref{bigtheorem}).   Conditions \ref{conditions} have already been dealt with in Theorem
\ref{preliminaryresult}, leaving checking hypotheses (i)---(iii) in the statement of Theorem \ref{bigtheorem}:

Fix $\lambda\in\Gamma$.  In the category of  $U_\zeta$-modules, we have\footnote{If $M$ is a $U({\mathfrak g}_K)$-module, then $M^{[1]}$ denotes the $U_\zeta$-module obtained by ``pulling $M$ back" through the
Frobenius morphism $U_\zeta\to U({\mathfrak g}_K)$. If $M=L_K(\lambda)$ is the irreducible ${\mathfrak g}_K$-module of
highest weight $\lambda$, then $M^{[1]}\cong L_K(p\lambda)$.}
 \begin{equation}\label{alternativegrading} P_K(\lambda)\cong Q_K(\lambda_0)\otimes L_K(\lambda_1)^{[1]}.\end{equation}
 Enlarging $\Gamma$ to a larger poset $\Lambda$ of weights (with $\Gamma$ an ideal
 in $\Lambda$), we can insure that $P_K(\lambda)$ in the display makes sense as an $\wA_K$-module. We assume this has been done.

Assuming $p\geq 2h-2$, there is an
$\wA$-module $\wQ(\lambda_0)$, a PIM in a certain highest weight category, lifting a corresponding truncated projective module for
$\wA/\pi\wA$, with $\wQ(\lambda_0)_K\cong Q_K(\lambda_0)$. (See \cite{DS}.)\footnote{Reference \cite{DS} discusses integral lifting of
 truncated projective modules. $Q(\lambda_0)$ is a PIM in the category of $G$-modules whose composition factors $L(\nu)$ all satisfy
  $\nu\leq 2(p-1)\rho+w_0\lambda_0$. This assumes $p\geq 2h-2$; see \cite{JanB}.}
  Also, $Q(\lambda_0):=\wQ(\lambda_0)/\pi\wQ(\lambda_0)$ is $\wfa_k$-projective, and
so $\wQ(\lambda_0)$, which is $\sO$-free, is $\widetilde\fa$-projective. Thus, $\wQ(\lambda_0)$
can be given an $\wfa$-grading, as
$$\wQ(\lambda_0)\cong\widetilde\fa\otimes_{\widetilde\fa_0}\wL(\lambda_0)\subseteq\fa_K\otimes_{\fa_{K,0}}L_K(\lambda_0),$$
where $\wL(\lambda_0)$ is a full $\wfa$-lattice in $L_K(\lambda_0)$ isomorphic to the projection
$\wQ(\lambda_0)/\wQ(\lambda_0)\cap \wQ(\lambda_0)_K$ in $L_K(\lambda_0)\cong
\wQ(\lambda_0)_K/\rad \wQ(\lambda_0)_K
\cong \wP_K(\lambda_0)/\rad P_K(\lambda_0)$.
 Fix a nonzero $\lambda$-weight vector $v=v_0\otimes v_1\in
Q_K(\lambda_0)_0\otimes L_K(\lambda_1)^{[1]}$, with $v_0\in Q_K(\lambda_0)_0=\left(A\otimes_{A_0}L_K(\lambda_0)\right)_0\cong A_0\otimes_{A_0}L_K(\lambda_0)\cong
L_K(\lambda_0)$. Put
\begin{equation}\label{integralprojectiveP}
\wP(\lambda)^\dagger=\wQ(\lambda_0)\otimes \wL^{\text{\rm min}}(\lambda_1)^{[1]},\end{equation}
 where $\wL^{\text{\rm min}}(\lambda_1)^{[1]}$ is the full $\widetilde U_\zeta$-lattice in $L_K(\lambda_1)^{[1]}$ generated
by $v_1$. (Of course, this lattice, as well as $\wQ(\lambda_0)$, must be viewed as $\widetilde U_\zeta$-modules to make sense of the tensor product.) In general,
$\wP(\lambda)^\dagger$ is not a projective $\wA$-module, though it is projective for $\wfa$.  Its scalar extension $\wP(\lambda)_K\cong P_K(\lambda)$ is $\wA_K$-projective.
The projectivity of $\wP(\lambda)^\dagger|_\wfa$ implies that it is tight. (See the observation following
Definition \ref{tightdef1}.) This establishes part (iii) of the hypothesis of Theorem \ref{bigtheorem}.

Also,
$\wP(\lambda)^\dagger=\wA v$, since adding $\wP(\lambda)^\dagger_{\geq 1}=\wQ(\lambda_0)_{\geq 1}\otimes\wL(\lambda_1)^{[1]}$ to both sides of $\wA v$ gives $\wP(\lambda)^\dagger$, as
may be checked in
$$(\wP(\lambda)^\dagger+\rad P_K(\lambda))/\rad P_K(\lambda)\cong\wL(\lambda_0)\otimes\wL^{\text{\rm min}}(\lambda_1)^{[1]}.$$
(Notice that the latter module is a homomorphic image of $\wDelta(\lambda)$. This follows from a result
of Z.~Lin, c.f., \cite[Prop.~1.9]{CPS7}.)   This establishes hypothesis
(i) of Theorem \ref{bigtheorem}.

To establish hypothesis (ii) of Theorem \ref{bigtheorem}, observe that $K\wP(\lambda)^\dagger\subseteq \wA_{K,0}v$.
Modulo $E_K(\lambda)$, $\wA_{K,0}v$ generates an $\wA_{K,0}$-submodule of $\wP_K(\lambda)/E_K(\lambda)\cong\Delta_K(\lambda)$. The latter module is completely reducible as an $A_{K,0}$-module, and the
$\wA_{K,0}$-module generated by its one-dimensional $\lambda$-weight space is isomorphic to
$L_K(\lambda)|_{\wA_{K,0}}$.  Thus, $K\wP(\lambda_0)+ E_K(\lambda)/E_K(\lambda)$ and
$\wA_{K,0}v+ E_K(\lambda)/E_K(\lambda)$ have the same dimension
and are thus equal (in view of the above containment). This establishes hypothesis (ii). \end{proof}

\section{Appendix I}  We prove the following result.

\begin{lemma} Suppose that $\Gamma$ is a poset ideal in a finite poset $\Lambda$, and that $B$
is a QHA over a field $K$ with weight poset $\Lambda$. Assume that $\gr B$ is also a QHA with
weight poset $\Lambda$, where we identify irreducible modules for $\gr B$ and $B$ through
$B/\rad B$.  Then $\gr (P(\gamma)_\Gamma)$ is a PIM for $(\gr B)_\Gamma$, for each $\gamma\in
\Gamma$. Here $P(\gamma)$ is the PIM for $B$ associated to $\gamma$. In particular, for
$\lambda\in\Lambda$, $\gr\Delta(\lambda)$
is the standard module for $\gr B$ associated to $\lambda$.\end{lemma}

\begin{proof} If $\Gamma=\Lambda$, there is nothing to prove. Otherwise, let $\mu\in\Lambda\backslash\Gamma$ be maximal in $\Lambda$, and put $m=[P(\gamma):L(\mu)]$. Thus, $P(\gamma)$ contains
a submodule $D\cong\Delta(\mu)^{\oplus m}$, and $P(\gamma)/D\cong P(\gamma)_{\Lambda'}$,
where $\Lambda'=\Lambda\backslash\{\mu\}$.

We have $D/\rad D\cong L(\mu)^{\oplus m}$,
 and $\gr^\# D / \gr^\# \rad D$ is a graded
version of $D/\rad D$, as may be checked by counting composition factors. Write
$$\gr^\#D/\gr^\#\rad D=\bigoplus_{i\in{\mathbb N}} L(\mu)(i)^{\oplus m_i},$$
with $m=\sum m_i$, $m_i\geq 0$.

Since $\mu$ is maximal, $\gr \Delta(\mu)$ is projective both as a graded or ungraded $\gr B$-module, and its
shifts $(\gr \Delta(\mu))(i)$ are also projective. Also, the head of $(\gr\Delta(\mu))(i)$ is $L(\mu)(i)$.
Consequently, there is a  graded module map
$$\phi: D:=\bigoplus_{i\in \mathbb N}\Delta(\mu)(i)^{\oplus m_i}\to\gr^\#D,$$
giving a
surjection when composed with $\gr^\#D\to\gr^\#D/\gr^\#\rad D$.  By the assumption that $\gr B$ is
a QHA,
the ungraded projective $\gr B$-module $\gr P(\gamma)$ has a filtration by standard modules, which may
be assumed to have bottom term $D'$, a direct sum of standard modules.  Consequently, if $(\gr^\# D)'$ denotes the ungrqaded $\gr B$-modules obtained by forgetting the grading on $\gr^\#D$,
we have
$$D'\subseteq{\text{\rm Image}}\,\phi'\subseteq (\gr^\# D)',$$
where $\phi'$ is the ungraded version of the map $\phi$. (Of course, the underlying image
is the same as for the graded version. A similar remark applies for $(\gr^\# D)'$.

Since
$$\dim D'=m\dim\Delta(\mu)=\dim D=\dim\gr^\#D,$$
it follows that equality holds in the inclusions
$$(\gr P(\gamma))_{\Lambda'}=\gr P(\gamma)/D'\cong\gr(P(\gamma)/D)=\gr(P(\gamma)_{\Lambda'})$$
as ungraded $\gr A$-modules. (Recall the graded isomorphism $\gr P(\gamma)/\gr^\#D\cong$ 
\linebreak
$(\gr P(\gamma)/D)$.)

The lemma now follows by induction on $\Lambda\backslash\Gamma$.
\end{proof}

As a corollary of the proof, we have the following result. In the statement, we maintain the
notation and hypotheses of the above lemma.   The corollary is parallel to Corollary \ref{Nameless}.

\begin{cor}Let $M$ be a (finite dimensional) $B$-module such that $M$ and $\gr M$ have
standard filtrations. Then $(\gr M)_\Gamma=\gr(M_\Gamma)$.\end{cor}
\begin{proof} This is obtained by essentially the same argument as above, using induction on
$\Lambda\backslash\Gamma$.\end{proof}

 \section{Appendix II: Quantum deformation theory over $\sO$} The goal of this appendix is provide a proof of
 Theorem \ref{preliminaryresult}(c). The proof will require some preparation and relies on
 results in \cite{AJS}. In the process, we extend (when $p>h$ is a prime) the quantum deformation
 theory of \cite{AJS} to a version over $\sO={\mathbb Z}_{(p)}[\zeta]$, where $\zeta$ is a $p$th root
 of unity.

  We let $\Phi$ be an irreducible root system as in \S5, and let $\zeta$ be a primitive $p$th
  root of unity. Assume that $p$ is an odd prime, $p\not=3$ in case $\Phi$ has type $G_2$. We also use other
  notation of \S5 above, e.~g., $d_\alpha=(\alpha,\alpha)/(\alpha_0,\alpha_0)$.  Notations
  of \cite{AJS} will be used as they arise. For instance, $U$ denotes the quantum algebra
  defined in \cite[\S1.3]{AJS} over ${\mathbb Q}(\zeta)$, a homomorphic image of the
 DeConcini-Kac quantum algebra
  $U_2$. Thus,
   $U=U^-U^0U^+$ is the quantum enveloping algebra defined on \cite[p.~15]{AJS}. Here $U^+, U^-$
  are finite dimensional, while $U^0$ is the algebra of Laurent polynomials in $K_\alpha$, $\alpha\in\Pi$.

   Let
 $$S':=\sO\left[\frac{K_\alpha-1}{\zeta^{d_\alpha} -1},\alpha\in\Pi\right].$$
 This is a polynomial algebra over $\sO$, an $\sO$-form of the polynomial algebra $K\left[K_\alpha, \alpha\in \Pi\right]$.
 Observe that, for any $\alpha\in\Phi^+$ (not just in $\Pi$), the expression
 $$H'_\alpha:=\frac{K_\alpha-1}{\zeta^{d_\alpha}-1}$$
 makes sense, using the definition of $K_\alpha$ as a product given on \cite[p.~47]{AJS} (see also
 (\ref{Kbeta})). A recursive
 argument, applying the identity $xy-1=(x-1)y + (y-1)$ to the numerator of $H'_\alpha$, shows that $H'_\alpha$
 belongs to $S'$. (Note that the denominator may be adjusted to have the correct $d_\alpha$ by multiplying by a unit
 in $\sO$, since all the numbers $d_\alpha$ are relatively prime to $p$.)

 We now make the following additional observations for $\alpha\in\Phi^+$, using the notation $[K_\alpha;j]$,
 $0\leq j< p$ on \cite[p.~48]{AJS}, and given implicitly below:

 \begin{itemize}
 \item[(1)] $K_\alpha\in S'$. (This is obvious, since $(\zeta^{d_\alpha}-1)\frac{K_\alpha-1}{\zeta^{d_\alpha} -1}
 =K_\alpha-1\in S'$, for $\alpha\in\Phi^+$.)

   \item[(2)] Let $\widehat{S'}$ be the completion (or just localization) with respect to the augmentation
 ideal of $S'$ (i.~e., the kernel of the augmentation map $S'\to\sO$. Then $K_\alpha^{-1}\in\widehat{S'}$.

  \item[(3)] $[K_\alpha; 0]\in K_\alpha^{-1}S'$, since
 $$[K_\alpha;0]=\frac{K_\alpha^{-1}}{\zeta^{-d_\alpha}}\cdot\frac{K_\alpha+1}{\zeta^{d_\alpha}+1}
\cdot\frac{K_\alpha-1}{\zeta^{d_\alpha}-1},$$
 and $p$ is odd (which implies $\zeta^{d_\alpha}+1$ is a unit in $\sO$).

 \item[(4)] $[K_\alpha;j]\in K_\alpha^{-1}S'$, $1\leq j< p$, since $[K_\alpha;j]=$
 $$\begin{aligned} 
& \frac{K_\alpha\zeta^{d_\alpha j}-\zeta^{-d_\alpha j}K_\alpha^{-1}}{\zeta^{d_\alpha}-\zeta^{d_\alpha}}\\[2mm]
 &=\frac{(K_\alpha-1)\zeta^{d_\alpha j} +\zeta^{-d_\alpha j}(1-K_\alpha^{-1})}{\zeta^{d_\alpha}-\zeta^{d_\alpha}  }
+ \frac{\zeta^{d_\alpha j} -\zeta^{-d_\alpha j}}{\zeta^{d_\alpha}-\zeta^{-d_\alpha}}\\[2mm]
&= \left(\frac {K_\alpha -1}{\zeta^{d_\alpha} -1}\right)\cdot\frac {1}{(\zeta^{d_\alpha} +1)\zeta^{-d_\alpha}}\dot(\zeta^{d_\alpha j} +\zeta^{-d_\alpha j} K_\alpha^{-1}) +\frac{\zeta^{d_\alpha j} -\zeta^{-d_\alpha j}}
{\zeta^{d_\alpha}-\zeta^{-d_\alpha}}.\end{aligned}
$$

\item[(5)] $[K_\alpha; j]^{-1}\in \widehat{S'}$ (or the corresponding localization of $S'$), as follows from
(2) and (4), $1\leq j<p$.

 \item[(6)] $\log K_\alpha\in \widehat {S'}$, since $\frac{(\zeta^{d_\alpha}-1)^r}{r}\in\sO$ for all
 integers $r\geq 1$.
 \end{itemize}

As mentioned above, the proofs in this section will require the results and methods of \cite{AJS}. In particular, we follow the notation
of that paper closely. Consider the algebras $S$ and $\sB_S$  \cite[p.~180]{AJS}  (expanded on
\cite[p.~179, bottom]{AJS} to allow direct sums of compositions of wall-crossing functors; indeed, we should
take $\sB_S:=\End^\#_{{\mathcal K}(W_a,S)}({\mathcal Q}_{\mathcal I}(S))$---see \cite[p.~219]{AJS}---for sufficiently large $\mathcal I$). The algebra $S$ is a symmetric algebra on the integral root lattice associated to $\Phi$
(the latter denoted $R$ in \cite{AJS}).
The algebra $\sB_S$ is $\mathbb Z$-graded, compatibly with a grading on $S$ in which every
root has grade 2. It  gives rise, by base change, to various algebras $\sB_A$, for any commutative $S$-algebra $A$. Taking $A={\mathbb Q}(\zeta)$ yields an algebra isomorphic to the ungraded endomorphism algebra of the ``$Y$-projective generator" of a certain category $\sC_K(\Omega)$, $K={\mathbb Q}(\zeta)$, with $\Omega$ any regular orbit of the affine Weyl group $W_a\cong W_p$ on integral weights. This result requires restrictions on $p$: it must be a positive odd integer $>h$, not divisible by $3$ if $\Phi$ is of type $G_2$. (In \cite{AJS}, $\Phi$ is allowed to be
decomposable.)
 Examination of this ungraded  endomorphism algebra shows that $\sB^{\text{\rm op}}_K$ is Morita equivalent to the block algebra of the small quantum group $u_\zeta$ associated to the orbit $\Omega$ . See \cite[Thm.~16.18]{AJS}, which is essentially already proved in this ``Case 2" by the Claim \cite[\S14.13(4)]{AJS}.

When $p>h$ is a prime, we can define the above Morita equivalence over a DVR $\sO$ with fraction
field $K$, namely, $\sO:={\mathbb Z}_{(p)}[\zeta]={\mathbb Z}[\zeta]_{(p,\zeta-1)}$. This is not done in \cite{AJS}, but, with $S'$
and $\widehat{S'}$ above in hand, it is possible to modify the arguments there to establish this
Morita equivalence. Here is an outline:
We wish to transport many of the results stated in \cite{AJS} for the algebra
$U$ over $K:={\mathbb Q}(\zeta)$ in ``Case 2" to an $\sO$-form $U_\sO$ of $U$. Like $U$, the form
$U_\sO$ has a triangular decomposition $U_\sO=U^-_\sO U^0_\sO U^+_\sO$ with $U^-_\sO$,
$U^+_\sO$ generated over $\sO$ by the $F_\alpha\,, E_\alpha$, $\alpha\in\Phi^+$, and
$$U^0_\sO:=\sO\left[K^{\pm1}_\sO, H'_\alpha,\alpha\in\Pi\right],$$
the localization of $S'$ above with respect to the multiplicatively closed set generated by the $K_\alpha$'s.
Using (3) and (4) above, and standard identities (e.g., Kac's identity \cite[Lem.~5.27]{DDPW}), we find
that $U_\sO$ is an $\sO$-algebra, an $\sO$-form of $U$. It also inherits the comultiplication on $U$, defined in \cite[\S7.1]{AJS}. Observe
$$\Delta(\frac{K_\alpha-1}{\zeta^{d_\alpha}-1})=\frac{K_\alpha-1}{\zeta^{d_\alpha}-1} \otimes K_\alpha
+ 1\otimes \frac{K_\alpha-1}{\zeta^{d_\alpha}-1}.$$

Define  $B_\sO$ to be the $\sO$-algebra obtained by localizing $U^0_\sO$ at the multiplicatively closed set
generated by all $[K_{\alpha};j]$, $\alpha\in\Phi^+$, $0\leq j<p$.
Then $B_\sO$ is an $\sO$-form of the $K$-algebra $B$ defined in \cite[p.~48, bottom]{AJS}. Put $H_\alpha=
[K_\alpha;0]$ as at the beginning of \cite[\S8]{AJS}. The results of \cite[\S8.6]{AJS} hold (with similar
proofs) if $B$ is replaced by
$B_\sO$. In
 particular, they hold for the $B_\sO$-algebra $A=\widehat{S'}$, which has other properties we need, such as those required on \cite[p.~103, top]{AJS}.
 Thus, the functor $\mathcal V$, given in \cite[p.~86]{AJS}, may be defined on the category
$\sC_A(\Omega)$, using for the $e_\beta(\lambda)$, on \cite[p.~86]{AJS}, generators of the
$\Ext^1$-groups described in terms of $A$ in \cite[\S\S 8.6, 8.7]{AJS}. The orbit $\Omega$ on \cite[p.~86]{AJS}
need not be regular, and this flexibility is important for discussing translation to, from, and through the
walls. The target category of $\mathcal V$ is a certain ``combinatorial category" ${\mathcal K}(\Omega,A)$
on which ``combinatorial translation functors" are defined. The ``combinatorial"
category ${\mathcal K}(\Omega,A) ={\mathcal K}(\Omega)$ is defined by analogy with
\cite[\S9.4,\S14.4]{AJS}.

(If $\Omega$ is regular, ${\mathcal K}(\Omega,A)$
is isomorphic to ${\mathcal K}(W_a, A)$ used above.) These translation functors are shown in \cite[Prop., \S10.11]{AJS} to
commute with the module theoretic translation functors. (Commutativity there implies commutativity in
our case.)

  Since $S'$ is a unique factorization
domain, so are $U^0_\sO$ and $B_\sO$ also unique factorization domains. In particular, the critical
intersection property
$$A=\bigcap_{\beta\in\Phi^+}A^\beta$$
holds for $A$ above, since it is flat over $B_\sO$. See \cite[Lem.~9.1]{AJS} and its proof.
The algebra $A^\beta$ is obtained by inverting all $H_\alpha=\frac{K_\alpha-K_\alpha^{-1}}{\zeta^{d_\alpha}-\zeta^{-d_\alpha}}\in B_\sO$ with $\beta\not=\alpha\in\Phi^+$, and setting $A^\beta=A\otimes_{B_\sO}B_\sO^\beta$. Now
\cite[Prop.~9.4]{AJS} essentially holds as stated,\footnote{In checking the analogues of \cite[\S9.4]{AJS} as well as the previous results in \S8 mentioned above, a good strategy is to read through these results and, when references are made to previous sections in \cite{AJS}, to look back at those references and check that they hold in our
context.   } and with the same proof, for any $A$ flat
over $B_\sO$, and, in particular, for $A=\widehat{S'}$.
This gives a fully faithful functor
$${\mathcal V}_\Omega:{\mathcal F}\sC_A(\Omega)\to{\mathcal K}(\Omega),$$
with $\mathcal{FC}_A(\Omega)$ denoting the category of $A$-flat $U_\sO$-modules, which are a
direct sum of weight spaces (``$X$-graded" in the sense of \cite{AJS}) with all weights in $\Omega$, and satisfying the conditions of \cite[\S2.3]{AJS} with $U$ replaced by $U_\sO$. 

The symmetric algebra $S=S({\mathbb Z}\Phi)$ is written in \cite[\S14.3]{AJS} using the symbol
$h_\alpha$ to denote a root $\alpha\in\Phi$. It is then given two different interpretations, as $\log K_\alpha$
and as $d_\alpha H_\alpha$ in ``Case 1" and  ``Case 2," respectively. We can essentially handle both
interpretations in our set-up at the same time: Fix $h_\alpha:=\log K_\alpha\in \widehat{S'}$ as in (6)
above. Then $d_\alpha H_\alpha$ differ from $h_\alpha$ only by multiplication by a unit in $\widehat{S'}$,
and a similar comparison may be made to the generators $H'_\alpha$ of $S'$. In particular, $\widehat{S'}$
is isomorphic to the completion of $S\otimes_{\mathbb Z}\sO$ with respect to its augmentation
ideal. So
$A=\widehat{S'}$ is flat over $S$, giving the nice base-change property of \cite[Lem.~14.8]{AJS} for passing from objects and morphisms of ${\mathcal K}(\Omega,S)$ to ${\mathcal K}(\Omega, A)$. Together with
the isomorphism ${\mathcal V}_\Omega$ above, this implies that ${\mathcal B}_S\otimes_SA$ is the endomorphism algebra of a projective generator of ${\mathcal C}_A(\Omega)$. Base changing
this generator from $A$ to $\sO$ shows, as in \cite[14.13(4)]{AJS}, that ${\mathcal B}_S\otimes_S\sO$ is the endomorphism algebra of a projective generator for a regular block of $\widetilde u_\zeta$-mod. See \cite[\S16.9]{AJS} for
a more complete treatment over $K$.

The algebra $\sB_S\otimes_S\sO$ retains the $\mathbb Z$-grading of $\sB_S$ and is compatible
with the grading of $\sB\otimes_S K$. The latter grading is shown to be positive and even Koszul
in \cite[\S\S18.17--8.21]{AJS}. (The hypothesis there on Lusztig conjecture holds for the quantum
case when $p>h$.) The positive grading on $\sB_S\otimes_S\sO$ transports to a positive grading
on $\widetilde u_\zeta$, with the required properties in (c). (The algebra $\widetilde u'_\zeta$ is Morita
equivalent to a product of copies of the algebra $(\sB_S\otimes_S \sO)^{\text{\rm op}}$, so it has the form $eMe$, where $M$ is a full matrix
algebra over a finite product of copies of the algebra $(\sB_S\otimes_S \sO)^{\text{\rm op}}$, and $e\in M$ is an idempotent such that $Me$ is a projective generator of
$M$-mod. Then $\gr eMe\cong e(\gr M) e\cong eMe$, the latter isomorphism coming from a similar one
for $(\sB_S\otimes_S \sO)^{\text{\rm op}})$, completing the proof.
\hskip2.5in  $\qed$

\vskip.4in

As a corollary of the proof, we have the following new result, independent of Theorem \ref{preliminaryresult}
and worthy of attention in its own right. The notations $u'_\zeta$ and $\widetilde u'_\zeta$ may
be found above Theorem \ref{preliminaryresult}. They refer to the ``regular" part of the small quantum
group $u_\zeta$ and of its integral form, respectively.

\begin{theorem}\label{newandimportant} Assume that $p>h$ is a prime. Then the algebra $\widetilde u'_\zeta$ over $\sO$ has a positive grading which base
changes to the Koszul grading on $u'_\zeta$ obtained in \cite[\S\S17-18 and p.~231]{AJS}.\end{theorem}

\bibliographystyle{amsplain}

\end{document}